\documentclass[a4paper,10pt]{article}
\pdfoutput=1

\usepackage[english]{babel}
\usepackage[noadjust]{cite}
\usepackage{amsfonts}
\usepackage{latexsym}
\usepackage{amsthm}
\usepackage{amsopn}
\usepackage{amsmath}
\usepackage[all]{xy}
\usepackage{verbatim}
\usepackage{amssymb}
\usepackage{mathrsfs}
\usepackage{float}
\usepackage[active]{srcltx}
\usepackage{fullpage}
\usepackage{manyfoot}
\usepackage{amscd}
\usepackage[dvips]{graphics}
\usepackage[dvips]{graphicx}
\usepackage{amscd}
\usepackage{color}
\usepackage{tikz}
\usetikzlibrary{arrows}
\usepackage[hang,flushmargin]{footmisc} 
\usepackage{dsfont}
\usepackage{capt-of}
\usetikzlibrary{matrix}
\usepackage[hmargin=3cm, vmargin=3cm]{geometry}
\usepackage{sansmath}
\usepackage{fancyvrb}
\usepackage{xcolor}

\usepackage{color, xcolor}
\usepackage[backref=page]{hyperref}

\renewcommand*{\backrefalt}[4]{%
    \ifcase #1 (Not cited.)%
    \or       (Cited on page {\textcolor{black}{#2}}.)
    \else     (Cited on pages {\textcolor{black}{#2}}.)
    \fi}

\newenvironment{psmallmatrix}
  {\left(\begin{smallmatrix}}
  {\end{smallmatrix}\right)}

\newtheorem{theorem}{Theorem}[section]
\newtheorem{proposition}[theorem]{Proposition}
\newtheorem{lemma}[theorem]{Lemma}
\newtheorem{corollary}[theorem]{Corollary}
\newtheorem*{main}{Main Theorem}

\theoremstyle{definition}
\newtheorem{remark}[theorem]{Remark}
\newtheorem{example}[theorem]{Example}
\newtheorem{definition}[theorem]{Definition}

\renewcommand{\to}{\longrightarrow}

\newcommand{\ZZ}{\mathbb{Z}}

\newcommand{\CC}{\mathbb{C}}
\newcommand{\PP}{\mathbb{P}}
\newcommand{\imh}{\textrm{Im}\; h}

\global\long\global\long\def\NS{{\rm NS}}

 \global\long\global\long\def\Kl{{\rm Kl}}
 \global\long\global\long\def\aut{{\rm Aut}}

\global\long\global\long\def\OO{\mathcal{O}}
\global\long\global\long\def\LL{\mathcal{L}}

 \global\long\global\long\def\l{\mathbb{\lambda}}
\global\long\global\long\def\s{\mathbb{\sigma}}

\definecolor{ccqqqq}{rgb}{0.8,0.,0.}
\definecolor{bcduew}{rgb}{0.7372549019607844,0.8313725490196079,0.9019607843137255}

\title{A pair of  rigid surfaces  \\ with $p_g=q=2$ and $K^2=8$ \\ whose universal cover is not the bidisk}
\author{Francesco Polizzi \and Carlos Rito \and Xavier Roulleau}

\date{}
\begin{document}
\maketitle
%


\begin{abstract}
We construct two  complex-conjugated rigid minimal surfaces with  $p_g=q=2$ and $K^2=8$ whose universal cover is not biholomorphic to the bidisk $\mathbb{H} \times \mathbb{H}$.
We show that these are the unique surfaces with these invariants and Albanese map of degree $2$, 
apart from the family of
product-quotient surfaces given in \cite{Pe11}.
This completes the classification of surfaces with $p_g=q=2, K^2=8$ and Albanese map of degree $2$.
\end{abstract}


\Footnotetext{{}}{\textit{$2010$ Mathematics Subject Classification}:
14J29, 14J10}

\Footnotetext{{}}{\textit{Keywords}: Surfaces of general type, Albanese map, double covers, abelian surfaces, rigid surfaces}


\setcounter{section}{-1}

\section{Introduction} \label{sec:intro}

Despite the work of many authors, minimal surfaces $S$ of general type with the lowest possible value of the holomorphic Euler characteristic, namely such that $\chi(\mathcal{O}_S)=1$, are far from being classified,
see e.g. the survey papers \cite{BaCaPi06}, \cite{BaCaPi11} and \cite{MP12} for a detailed bibliography on the subject.
These surfaces satisfy the Bogomolov-Miyaoka-Yau inequality  $K^2\leq 9.$ 

The ones with $K^2=9$ are rigid, their universal cover is the unit ball in $\CC^2$ and  $p_g=q\leq 2$. 
The fake planes, i.e. surfaces  with $p_g=q=0$, have been classified in \cite{PrYe} and \cite{CaSt}, the two Cartwright-Steger surfaces  \cite{CaSt} satisfy $q=1$, whereas no example is known for $p_g=q=2$.

The next case is 
$K^2=8$. In this situation,  Debarre's inequality for irregular surfaces \cite{Deb81} implies 
\begin{equation*}
0\leq p_g=q\leq 4.
\end{equation*}
The cases $p_g=q=3$ and $p_g=q=4$ are nowadays classified (\cite{HacPar}, \cite{Piro}, \cite[Beauville's appendix]{Deb81}), whereas for $p_g=q\leq 2$  some families are known (\cite{Sh78}, \cite{ BCG05}, \cite{ Pol06}, \cite{ Pol08}, \cite{ CP09}, \cite{ Pe11}) but there is no complete description yet.

All the examples of minimal surfaces with $\chi=1$ and $K^2=8$ known so far are uniformized by the bidisk $\mathbb H\times\mathbb H$, where $\mathbb H = \{z \in 	\mathbb{C} \; | \; \mathrm{Im}\, z >0 \}$ is the Poincar\'e upper half-plane; so the following question naturally arose:

\smallskip
\emph{Is there a smooth minimal surface of general type with invariants $\chi=1$ and $K^2=8$ and whose universal cover is not biholomorphic to $\mathbb H\times\mathbb H$?}
\smallskip

For general facts about surfaces uniformized by the bidisk, we refer the reader to \cite{CaFra}. One of the aims of this paper is to give an affirmative answer to the question above. In fact we construct two rigid, minimal surfaces with $p_g=q=2$ and $K^2=8$ whose universal cover is not the bidisk. 
Moreover, we show that these surfaces are complex-conjugated and that they are the unique minimal surfaces with these invariants and having Albanese map of degree $2$, apart from the family of product-quotient surfaces constructed in \cite{Pe11}.
This complete the classification of minimal surfaces with $p_g=q=2$, $K^2=8$ and Albanese map of degree $2$.

Our results can be summarized as follows, see Proposition \ref{prop:branch-locus}, Theorem \ref{thm:class-type I}, Theorem \ref{thm:class-type II},
Theorem \ref{thm:The-number-of-Surfaces} and Proposition \ref{prop:univ-II}.
\begin{main}
Let $S$ be a smooth, minimal surface of general type with $p_g=q=2$, $K^2=8$ and such that its Albanese map $\alpha \colon S \to A:=\mathrm{Alb}(S)$ is a generically
finite double cover. Writing $D_A$ for the branch locus of $\alpha$, there are exactly two possibilities, both of which occur$:$ 
\begin{itemize}
\item[$\boldsymbol{(I)}$] $D_A^2=32$ and $D_A$ is an irreducible curve with one ordinary point of multiplicity $6$ and no other singularities. These are the product-quotient surfaces constructed in \emph{\cite{Pe11}}$;$
\item[$\boldsymbol{(II)}$] $D_A^2=24$ and $D_A$ has two ordinary points $p_1$, $p_2$ of multiplicity $4$ and no other singularities.
More precisely, in this case we can write 
\begin{equation*}
D_A = E_1+E_2+E_3+E_4,
\end{equation*}
where the $E_i$ are elliptic curves intersecting pairwise transversally at $p_1$, $p_2$ and not elsewhere. Moreover, $A$ is an \'etale double cover of the abelian surface $A':=E' \times E'$, where $E'$ denotes the equianharmonic elliptic curve. 

Up to isomorphism, there are exactly two such surfaces, which are complex-conjugate. Finally, the universal cover of these surfaces is not biholomorphic to the bidisk $\mathbb H\times\mathbb H.$
\end{itemize}
\end{main}
According to the dichotomy in the Main Theorem, we will use the terminology \emph{surfaces of type I} and \emph{surfaces of type II}, respectively. Besides answering the question above about the universal cover, the Main Theorem is also significant because 
\begin{itemize} 
\item it contains a new geometric construction of rigid surfaces, which is usually something hard to do;
\item it provides a substantially new piece in the fine classification of minimal surfaces of general type with $p_g=q=2$; 
\item it shows that surfaces of type $II$ present the the so-called  \verb|Diff|$\nRightarrow$\verb|Def| phenomenon, meaning that their diffeomorphism type  does not determine their deformation class, see Remark \ref{rem:def-diff}.
\end{itemize}
Actually, the fact that there is exactly one surface of type $II$ up to complex conjugation is a remarkable feature. The well-known Cartwright-Steger surfaces \cite{CaSt} share the same property, however our construction is of a different nature, more geometric and explicit.
 
The paper is organized as follows.

In Section \ref{sec:Albanese} we provide a general result for minimal surfaces $S$ with $p_g=q=2$, $K^2=8$ and Albanese map $\alpha \colon S \to A$ of degree $2$, and we classify all the possible branch loci $D_A$ for $\alpha$ (Proposition \ref{prop:branch-locus}).

In Section \ref{sec:type I} we consider surfaces of type $I$, showing that they coincide with the family of  
product-quotient surfaces constructed in \cite{Pe11} (Theorem \ref{thm:class-type I}).

In Section \ref{sec:type II} we start the investigation of surfaces of type $II$. The technical core of this part is Proposition \ref{prop:2-divisibility}, showing that, in this situation, the pair $(A, \, D_A)$ can be realized as an \'etale double cover of the pair $(A', \, D_A')$, where $D_{A'}$ is a configuration of four elliptic curves in $A'=E' \times E'$ intersecting pairwise and transversally only at the origin $o' \in A'$ (as far as we know, the existence of such a configuration was first remarked in \cite{Hir84}). The most difficult part is to prove that we can choose the double cover $A \to A'$ in such a way that the curve $D_A$ becomes $2$-divisible in the Picard group of $A$ (Proposition \ref{prop:2-divisibility}). The rigidity of $S$ then follows from a characterization of $A'$ proven in \cite{KH04} (cf. also \cite{Aide}). 

In Section \ref{sec:classification-II} we show that there are precisely two surfaces of type $II$ up to isomorphism, and that they are complex-conjugated (Theorem \ref{thm:The-number-of-Surfaces}). In order to do this, we have to study the groups of automorphisms and anti-biholomorphisms of $A$ that preserve the branch locus $D_A$, and their permutation action on the set of the sixteen square roots of $\mathcal{O}_A(D_A)$ in the Picard group of $A$ (Proposition \ref{prop:permutation-action}). 

Finally, we show that the universal cover of the surfaces of type $II$
is not biholomorphic to $\mathbb{H} \times \mathbb{H}$ (Proposition \ref{prop:univ-II} and Remark \ref{rem:Shimura}), we note that they can be given the structure of an open ball quotient in at least two different ways (Remark \ref{openball}) and we sketch an alternative geometric construction for their Albanese variety $A$ (Remark \ref{remGeomConstr}).

\bigskip 
\noindent\textbf{Acknowledgments.}
F. Polizzi was partially supported by GNSAGA-INdAM.
C. Rito was supported by FCT (Portugal) under the project PTDC/MAT-GEO/2823/2014,
the fellowship SFRH/ BPD/111131/2015 and by CMUP (UID/MAT/00144/2013),
which is funded by FCT with national (MEC) and European structural funds through the programs FEDER, under the partnership agreement PT2020.
We thank M. Bolognesi, F. Catanese, F. Lepr\'evost,  B. Poonen, R. Pardini, C. Ritzenthaler and M. Stoll for useful conversations and suggestions, and the MathOverflow user Ulrich for his answer in the thread:
\verb|http://mathoverflow.net/questions/242406|\\
This work started during the workshop \emph{Birational Geometry of Surfaces}, held at the Department of Mathematics of the University of Roma Tor Vergata from January 11 to January 15, 2016. We warmly thank the organizers for the invitation and the hospitality.

Finally, we are	indebted to the anonymous referees for their helpful comments and remarks, that considerably improved the presentations of these results.
\bigskip

\noindent\textbf{Notation and conventions.} We work over the field of complex numbers. All varieties are assumed to be projective. 
For a smooth surface $S$, $K_S$ denotes the \emph{canonical class}, $p_g(S)=h^0(S, \, K_S)$ is the \emph{geometric genus},
$q(S)=h^1(S, \, K_S)$ is the \emph{irregularity} and $\chi(\mathcal{O}_S)=1-q(S)+p_g(S)$ is the \emph{Euler-Poincar\'e characteristic}.

Linear equivalence of divisors is denoted by $\simeq$.
If $D_1$ is an effective divisor on $S_1$ and $D_2$ is an effective divisor on $S_2$, we say that the pair $(S_1, \, D_1)$ is an \emph{\'etale double cover} of the pair $(S_2, \, D_2)$ if there exists an \'etale double cover $f \colon S_1 \to S_2$ such that $D_1 = f^* D_2$. 

If $A$ is an abelian surface, we denote by $(-1)_A \colon A \to A$ the involution $x \to -x$. If $a \in A$, we write $t_a \colon A \to A$ for the translation by $a$, namely $t_a(x)=x+a$. We say that a divisor $D \subset A$ (respectively, a line bundle $\mathcal{L}$ on $A$) is \emph{symmetric} if $(-1)_A^*D = D$ (respectively, if $(-1)_A^* \mathcal{L} \simeq \mathcal{L}$).

\section{The structure of the Albanese map} \label{sec:Albanese}

Let us denote by $S$ a minimal surface of general type with $p_g=q=2$ and maximal Albanese dimension, and by $$\alpha \colon S \to A=\textrm{Alb}(S)$$ its Albanese map. It follows from \cite[Section 5]{Ca13} that $\deg \alpha$ is equal to the index of the image of
$\wedge^4 H^1(S, \, \mathbb{Z})$ inside $H^4(S, \, \mathbb{Z}) = \mathbb{Z}[S]$, hence it is a topological invariant of $S$.
So, one can try to classify these surfaces by looking at the pair of invariants $\left(K_S^2, \, \deg \alpha\right)$.

\begin{lemma} \label{lem:alb-isom}
Let $S$ be as above and assume that there is a generically  finite double cover $\tilde{\alpha} \colon S \to \widetilde{A}$, where $\widetilde{A}$ is an abelian surface. Then  
$\widetilde{A}$ can be identified with $A=\mathrm{Alb}(S)$ and there exists an automorphism $\psi \colon A \to A$ such that $\tilde{\alpha} = \psi \circ \alpha$.
\end{lemma}
\begin{proof}
The universal property of the Albanese map (\cite[Chapter V]{Be96}) implies that the morphism $\widetilde{\alpha} \colon S \to \widetilde{A}$ factors through a morphism $\psi \colon A \to \widetilde{A}$. But $\widetilde{\alpha}$ and $\alpha$ are both generically of degree $2$, so $\psi$ must be a birational map between abelian varieties, hence an isomorphism. Thus we can identify $\widetilde{A}$ with $A$ and with this identification $\psi$ is an automorphism of $A$.
\end{proof}

Throughout the paper, we will assume $\deg \alpha=2$, namely that $\alpha \colon S \to A$ is a generically finite  double cover. Let us denote by $D_A \subset A$ the branch locus of $\alpha$ and let 
\begin{equation} \label{dia.alpha}
\xymatrix{
S \ar[r]^{c} \ar[dr]_{\alpha} & X \ar[d]^{\alpha_X} \\
 & A}
\end{equation}
be its Stein factorization. The map $\alpha_X \colon X \to A$ is a finite double cover and the fact that $S$ is smooth implies that $X$ is normal, see \cite[Chapter I, Theorem 8.2]{BHPV03}. 
In particular $X$ has at most isolated singularities, hence $D_A$ is reduced. Moreover, $D_A$ is $2$-divisible in $\mathrm{Pic}(A)$, in other words there exists a divisor $L_A$ on $A$ such that $D_A \simeq 2L_A$.

We have a \emph{canonical resolution diagram}
\begin{equation} \label{dia.resolution}
\begin{CD}
\bar{S}  @> >> X\\
@V{\beta}VV  @VV \alpha_X V\\
B @> {\varphi}>> A,\\
\end{CD}
\end{equation}
see \cite[Chapter III, Section 7]{BHPV03}, \cite[Section 2]{PePol13b} and \cite{Ri10}. Here $\beta \colon \bar{S} \to B$ is a finite double cover, $\bar{S}$ is smooth, but not necessarily minimal, $S$ is the minimal model of $\bar{S}$ and $\varphi \colon B \to A$ is composed of a series of blow-ups. Let
$x_1, \, x_2, \ldots, x_r$ be the centers of these blow-ups and $\mathsf{E}_1, \ldots, \mathsf{E}_r$ 
the reduced strict transforms in $B$ of the corresponding exceptional divisors. Then the scheme-theoretic inverse image $\mathcal{E}_j$ of $x_j$ in $B$ is a linear combination with non-negative, integer coefficients of the curves in the set $\{\mathsf{E}_i\}$, and we have 
\begin{equation*}
\mathcal{E}_i \mathcal{E}_j=- \delta_{ij}, \quad K_B=\varphi^*
K_A+\sum_{i=1}^r \mathcal{E}_i. 
\end{equation*}
Moreover the branch locus $D_B$ of $\beta \colon \bar{S} \to B$ is smooth and can be written as
\begin{equation} \label{hat B}
D_B=\varphi^* D_A- \sum_{i=1}^r d_i \mathcal{E}_i,
\end{equation}
where the $d_i$  are even positive integers, say $d_i=2m_i$. Motivated  by \cite[p. 724-725]{Xi90} we introduce the following definitions:
\begin{itemize}
\item a \emph{negligible singularity} of $D_A$ is a point $x_j$ such that $d_j=2$, and $d_i \leq 2$ for any point $x_i$ infinitely near to $x_j;$
\item a $[2d+1, \, 2d+1]$-\emph{singularity} of $D_A$ is a pair $(x_i, \, x_j)$ such that $x_i$ belongs to the first infinitesimal neighbourhood of $x_j$ and $d_i=2d + 2$,  $d_j=2d$.
\item a $[2d, \, 2d]$-\emph{singularity} of $D_A$ is a pair $(x_i, \, x_j)$ such that $x_i$ belongs to the first infinitesimal neighbourhood of $x_j$ and $d_i=d_j=2d$;
\item a \emph{minimal singularity} of $D_A$ is a point $x_j$ such that its inverse image in $\bar{S}$ via the canonical resolution contains no $(-1)$-curves.
\end{itemize}
Let us give some examples:
\begin{itemize}
\item an ordinary double point and an ordinary triple point are both minimal, negligible singularities. More generally, an ordinary $d$-ple point is a always a minimal singularity, and it is non-negligible for $d \geq 4$; 
\item  a $[3, \, 3]$-point (triple point $x_j$ with a triple point $x_i$ in its first infinitesimal neighbourhood) is neither minimal nor negligible. Indeed, in this case we have 
\begin{equation*} 
D_B=\varphi^*D_A-2 \mathsf{E}_j - 6 \mathsf{E}_i = \varphi^*D_A - 2 \mathcal{E}_j - 4 \mathcal{E}_i,
\end{equation*}
with $\mathcal{E}_j=\mathsf{E}_j+\mathsf{E}_i$ and  $\mathcal{E}_i=\mathsf{E}_i$. The divisor $\mathsf{E}_j$ is a $(-2)$-curve contained in the branch locus of $\beta \colon \bar{S} \to B$, so its pull-back in $\bar{S}$ is a $(-1)$-curve;
\item a $[4, \, 4]$-point (quadruple point $x_j$ with a quadruple point $x_i$ in its first infinitesimal neighbourhood) is minimal and non-negligible. Indeed, in this case we have 
\begin{equation*} 
D_B=\varphi^*D_A- 4 \mathsf{E}_j - 8 \mathsf{E}_i = \varphi^*D_A - 4 \mathcal{E}_j - 4 \mathcal{E}_i,
\end{equation*}
with $\mathcal{E}_j=\mathsf{E}_j+\mathsf{E}_i$ and $\mathcal{E}_i=\mathsf{E}_i$. The divisor $\mathsf{E}_j$ is a $(-2)$-curve that does not intersect the branch locus of $\beta \colon \bar{S} \to B$, so its pull-back in $\bar{S}$ consists of the disjoint union of two $(-2)$-curves, that are the unique rational curves coming from the canonical resolution of the singularity.
\end{itemize}

Let us come back now to our original problem.
\begin{lemma} \label{lem:minimal}
In our situation, the following holds$:$
\begin{itemize}
\item[$\boldsymbol{(a)}$] we have $S= \bar{S}$ in \eqref{dia.resolution} if and only if all singularities of $D_A$ are minimal$;$
\item[$\boldsymbol{(b)}$] if $S$ contains no rational curves, then $D_A$ contains no negligible singularities. 
\end{itemize}
\end{lemma}
\begin{proof}
\begin{itemize}
\item[$\boldsymbol{(a)}$] If $D_A$ contains a non-minimal singularity then, by definition, $\bar{S}$ is not a minimal surface, hence $\bar{S} \neq S$.
Conversely, if all singularities of $D_A$ are minimal then there are no $(-1)$-curves on $\bar{S}$ coming from the resolution of the singularities of $D_A$. Since the abelian surface $A$ contains no rational curves, this implies that $\bar{S}$ contains no $(-1)$-curves at all, so $\bar{S}=S$.  
\item[$\boldsymbol{(b)}$] Any negligible singularity of $D_A$ is minimal and gives rise to some rational double point in $X$, and hence to some $(-2)$-curve in $\bar{S}$ that cannot be contracted by the blow-down morphism $\bar{S} \to S$ (since $A$ contains no rational curves, it follows as before that all $(-1)$-curves in $\bar{S}$ come from the resolution of singularities of $D_A$). This is impossible because we are assuming that $S$ contains no rational curves.
\end{itemize}
\end{proof}
By using the formulae in \cite[p. 237]{BHPV03}, we obtain
\begin{equation} \label{eq:sum-sing}
2=2 \chi(\mathcal{O}_{\bar{S}})=L_A^2- \sum m_i(m_i-1), \quad
K_{\bar{S}}^2=2L_A^2-2 \sum (m_i-1)^2.
\end{equation}
Notice that the sums only involve the  non-negligible singularities of $D_A\simeq 2L_A$. The two equalities in \eqref{eq:sum-sing} together imply
\begin{equation} \label{eq:res.can} 
K_S^2 \geq K_{\bar{S}}^2 = 4 + 2 \sum (m_i-1).
\end{equation}

We are now ready to analyse in detail the case $K_S^2=8$.
\begin{proposition} \label{prop:branch-locus}
Let $S$ be a minimal surface with $p_g=q=2$ and $K_S^2=8$. Then $S$ contains no rational curves, in particular $K_S$ is ample.  Using the previous notation, if the Albanese map $\alpha \colon S \to A$ is a generically finite double cover then we are in one of the following cases$:$
\begin{itemize}
\item[$\boldsymbol{(I)}$] $D_A^2=32$ and $D_A$ has one ordinary singular point  of multiplicity $6$ and no other singularities$;$ 
\item[$\boldsymbol{(II)}$] $D_A^2=24$ and $D_A$ has two ordinary singular points  of multiplicity $4$ and no other singularities.
\end{itemize}
\end{proposition}
\begin{proof}
The non-existence of rational curves on $S$ is a consequence of a general bound for the number of rational curves on a surface of general type, see \cite[Proposition 2.1.1]{Miy}.  

Since $K_S^2=8$, inequality \eqref{eq:res.can} becomes
\begin{equation} \label{eq:res.can-1}
\sum(m_i - 1) \leq 2.
\end{equation}
By Lemma \ref{lem:minimal} there are no negligible singularities in $D_A$, so  \eqref{eq:res.can-1} implies that we have three possibilities:
\begin{itemize}
\item $D_A$ contains precisely one singularity (which is necessarily ordinary) and $m_1=3$, that is $d_1=6$; this is case $\boldsymbol{(I)}.$ 
\item $D_A$ contains precisely two singularities and $m_1=m_2=2$, that is $d_1=d_2=4$. We claim that these two quadruple points cannot be infinitely near. In fact, the canonical resolution of a $[4, \, 4]$-point implies that $\bar S$ contains (two) rational curves and, since a $[4, \, 4]$-point is a minimal singularity, this would imply the existence of rational curves on $S=\bar{S}$, a contradiction.
So we have two ordinary points of multiplicity $4$, and we obtain case  $\boldsymbol{(II)}.$ 
\item $D_A$ contains precisely one singularity (which is necessarily ordinary) and $m_1=2$, that is $d_1=4$. An ordinary singularity is minimal, hence we get equality in \eqref{eq:res.can}, obtaining $K_S^2=6$ (this situation is considered in \cite{PePol13b}), that is a contradiction.
\end{itemize}   
\end{proof}

\begin{remark} \label{rem:sing-Stein}
Lemma \ref{lem:minimal} and Proposition \ref{prop:branch-locus} imply that for any surface $S$ with $p_g=q=2$, $K_S^2=8$ and Albanese map of degree $2$, we have $\bar{S}=S$. Furthermore, referring to diagram \eqref{dia.alpha}, the following holds: 
\begin{itemize}
\item in case $\boldsymbol{(I)}$, the birational morphism $c \colon S \to X$ contracts precisely one smooth curve $Z$, such that $g(Z)=2$ and $Z^2=-2$. This means that the singular locus of $X$ consists of one isolated singularity $x$, whose geometric genus is $p_g(X, \, x)= \dim_{\mathbb{C}}R^1c_* \mathcal{O}_S = 2$;
\item in case $\boldsymbol{(II)}$, the  birational morphism $c \colon S \to X$ contracts precisely two disjoint elliptic curves $Z_1, \, Z_2$ such that $(Z_1)^2 = (Z_2)^2=-2$. This means that the singular locus of $X$ consists of two isolated elliptic singularities $x_1, \, x_2$ of type $\widetilde{E}_7$, see \cite[Theorem 7.6.4]{Is14}.
\end{itemize}
\end{remark}

\begin{definition} \label{def:typeI-II}
According to the dichotomy in Proposition \ref{prop:branch-locus}, we will use the terminology \emph{surfaces of type I} and \emph{surfaces of type II}, respectively.
\end{definition}

\begin{proposition} \label{prop:branch-type-II}
Let us denote as above by $D_A$ the branch locus of the Albanese map $\alpha \colon S \to A$. Then$:$
\begin{itemize}
\item if $S$ is of type $I,$ the curve $D_A$ is irreducible$;$
\item if $S$ is of type $II,$ the curve $D_A$ is of the form
$D_A = E_1+E_2+E_3+E_4,$
where the $E_i$ are elliptic curves meeting pairwise transversally at two points $p_1$, $p_2$  and not elsewhere. In particular, we have $E_iE_j=2$ for $i \neq j.$ 
\end{itemize}
\end{proposition}
\begin{proof}
Suppose first that $S$ is of type $I$ and consider the blow-up $\varphi \colon B \to A$ at the singular point $p \in D_A$. Let $C_1,\ldots,C_r$ be the irreducible components of the strict transform of $D_A$ and $\mathcal{E} \subset B$ the exceptional divisor.
The curve $D_A$ only contains the ordinary singularity $p$, so the $C_i$ are pairwise disjoint; moreover, the fact that 
\begin{equation*}
\sum_{i=1}^r C_i =\varphi^*D_A- 6 \mathcal{E} 
\end{equation*}
is  $2$-divisible in $\mathrm{Pic}(B)$ 
implies that $C_i^2=C_i \big( \sum_{i=1}^r C_i \big) $ is an even integer.
Let us recall now that the abelian surface $A$ contains no rational curves, so $g(C_i) >0$ for all $i \in \{1, \ldots, r\}$. On the other hand, if $g(C_i)=1$ then its image $D_i:=\varphi(C_i)$ is a smooth elliptic curve, because it is a curve of geometric genus $1$ on the abelian surface $A$. Thus $D_i^2=0$ and, since $p \in D_i$, it follows $C_i^2=-1$, a contradiction because $-1$ is an odd integer. So we infer $g(C_i)\geq 2$ for all $i \in \{1, \ldots, r\}$  and we can write
\begin{equation*}
\begin{split}
6-4 & = \mathcal{E}(\varphi^* D_A - 6 \mathcal{E}) + (\varphi^*D_A - 6 \mathcal{E})^2 \\ & =  K_B \bigg(\sum_{i=1}^r C_i \bigg) + \bigg( \sum_{i=1}^r C_i \bigg)^2 =  \sum_{i=1}^r (2 g(C_i)- 2) \geq 2r,
\end{split}
\end{equation*}
that is $r=1$ and $D_A$ is irreducible.

Assume now that $S$ is of type $II$, and write $D_A = E_1 + \cdots + E_r,$ where each $E_i$ is an irreducible curve. Denote by $m_i$ and $n_i$ the multiplicities of $E_i$ at the two ordinary singular points $p_1$ and $p_2$ of $D_A$, and let $p_a(E_i)$ and $g_i$ be the arithmetic and the geometric genus of $E_i$, respectively.
We have
$\sum_{i=1}^r m_i = \sum_{i=1}^r n_i =4$
and
\begin{equation*}  \label{eq:Ci-2}
E_i^2 = 2p_a(E_i)-2 = 2g_i -2 + m_i(m_i-1)+n_i(n_i-1).
\end{equation*}
Using this, we can write
\begin{equation*}
\begin{split}
24 & = D_A^2 = \sum_{i=1}^r E_i^2 + 2 \sum_{j<k} E_j E_k \\
& = 2 \sum_{i=1}^r g_i - 2r + \sum_{i=1}^r m_i(m_i-1) + \sum_{i=1}^r n_i(n_i-1) + 2 \sum_{j<k} (m_j m_k + n_j n_k) \\
& = 2 \sum_{i=1}^r g_i - 2r  + \bigg(\sum_{i=1}^r m_i \bigg)^2 + \bigg(\sum_{i=1}^r n_i \bigg)^2 - \sum_{i=1}^r m_i  - \sum_{i=1}^r n_i \\
& = 2 \sum_{i=1}^r g_i - 2r + 24, 
\end{split}
\end{equation*}
that is 
$\sum_{i=1}^r g_i = r.$
Since $A$ contains no rational curves we have $g_i \geq 1$, and we conclude that 
\begin{equation} \label{eq:Ci-4}
g_1 =  \cdots = g_r=1.
\end{equation}
But every curve of geometric genus $1$ on $A$ is smooth, so \eqref{eq:Ci-4} implies that $D_A$ is the sum of $r$ elliptic curves $E_i$ passing through the singular points $p_1$ and $p_2$. Therefore $r=4$, because these points have multiplicity $4$ in the branch locus $D_A$.
\end{proof}

\section{Surfaces of type $I$} \label{sec:type I}

\subsection{The product-quotient examples} \label{2sec:type I}
The following family of examples, whose construction can be found in \cite{Pe11}, shows that surfaces of type $I$ do actually exist. Let $C'$ be a curve of genus $g(C') \geq 2$ and let $G$ be a finite group that acts freely on $C'\times C'$. We assume moreover that the action is \emph{mixed}, namely that there exists an element in $G$ exchanging the two factors; this means that
\begin{equation*}
G\subset {\rm Aut}(C' \times C') \simeq {\rm Aut}(C')^2 \rtimes\mathbb{Z}/2 \mathbb{Z}
\end{equation*}
is not contained in ${\rm Aut}(C')^2$. Then the quotient $S:=(C'\times C')/G$ is a smooth surface with 
\begin{equation} \label{eq:inv-prod-quot}
\chi(\mathcal{O}_S)=(g-1)^2/|G|, \quad K_S^2=8\chi (\mathcal{O}_S).
\end{equation}
The intersection $G^0:=G\,\cap\,{\rm Aut}(C')^2$ is an index $2$ subgroup of $G$ (whose action on $C'$ is independent on the factor).
From \cite[Theorem 3.6, Theorem 3.7 and Lemma 3.9]{Fr13}, the exact sequence
\begin{equation*}
1 \to G^0 \to G \to \mathbb{Z}/ 2 \mathbb{Z} \to 1
\end{equation*}
is non-split and the genus of the curve $C := C' /G^0$ equals $q(S)$.

We have a commutative diagram
\begin{equation*}
\begin{CD} 
C' \times C' @> t >> C\times C\\ @V VV  @VV u V\\ S@> \beta >> \textrm{Sym}^2(C),
\end{CD}
\end{equation*}
where $t \colon C' \times C' \to C \times C $ is a $(G^0\times G^0)$-cover, $u \colon C \times C \to \textrm{Sym}^2(C) $ is the natural projection onto the second symmetric product and $\beta \colon S \to \textrm{Sym}^2(C)$ is a finite cover of degree $|G^0|$.

Assume now that $C'$ has genus $3$ and that  $G^0 \simeq \mathbb{Z}/2 \mathbb{Z}$ (hence $G \simeq \mathbb{Z}/4 \mathbb{Z}$).
Since $G$ acts freely on $C'\times C',$ then $G^0$ acts freely on $C'$ and thus $C$ has genus $2$.
Denoting by $\Delta \subset C \times C$ the diagonal and by $\Gamma \subset C \times C$ the graph of the hyperelliptic involution $\iota \colon C \to C$, we see that $\Delta$ and $\Gamma$ are smooth curves isomorphic to $C$ and satisfying
\begin{equation*}
\Delta \Gamma =6, \quad \Delta^2 = \Gamma^2 = -2.
\end{equation*}
The ramification divisor of $u$ is precisely $\Delta$, so $u(\Delta)^2=-4$, whereas $u(\Gamma)$ is a
$(-1)$-curve. The corresponding blow-down morphism $\varphi \colon \textrm{Sym}^2(C) \to A$ is the Abel-Jacobi map, and $A$ is an abelian surface isomorphic to the Jacobian variety $J(C)$. The composed map 
\begin{equation*}
\alpha= \varphi \circ \beta \colon S \to A  
\end{equation*}  
is a generically finite double cover, that by the universal property coincides, up to automorphisms of $A$, with the Albanese morphism of $S$. Such a morphism is branched over $D_A := (\varphi \circ u)(\Delta)$, which is a curve with $D_A^2=32$ and containing an ordinary sextuple point and no other singularities: in fact, the curves $u(\Delta)$ and $u(\Gamma)$ intersect transversally at precisely six points, corresponding to the six Weierstrass points of $C$.

From this and \eqref{eq:inv-prod-quot}, it follows that $S$ is a surface with $p_g=q=2$, $K_S^2=8$ and of type $I$. Note that, with the notation of Section \ref{sec:Albanese}, we have $B= \textrm{Sym}^2(C)$ and $D_B = u(\Delta)$.

\begin{remark} \label{rem:curve-6-ple}
Here is a different construction of the singular curve $D_A$ considered in the previous example. Let $A:=J(C)$ be the Jacobian of a smooth genus $2$ curve and let us consider a symmetric theta divisor $\Theta \subset A$. Then the Weierstrass points of  $\Theta$ are six $2$-torsion points of $A$, say $p_0, \ldots, p_{5}$, and $D_A$ arises as the image of $ \Theta $ via the multiplication  map $2_A \colon A \to A$ given by $x \mapsto 2x$. Note that $D_A$ is numerically equivalent to $4 \Theta$.
\end{remark}

\begin{remark} \label{rem:pign-pol}
Recently, R. Pignatelli and the first author studied some surfaces with $p_g=q=2$ and $K_S^2=7$, originally constructed in   \cite{CanFr15} and arising as \emph{triple} covers $S \to A$ branched over $D_A$, where $(A, \, D_A)$ is as in the previous example. We refer the reader to \cite{PiPol16} for more details.  
\end{remark}

\subsection{The classification} \label{subsec:classification I}

The aim of this subsection is to show that every surface of type $I$ is a product-quotient surface of the type described in Subsection \ref{2sec:type I}.

\begin{lemma}\label{symmetric}
Let $D$ be an irreducible curve contained in an abelian surface $A$, with $D^{2}=32$ and having an ordinary point $p$ of multiplicity $6$ and no other singularities. 
Then, up to translations, we can suppose $p=0$ and $D$ symmetric, namely $(-1)_A^* D=D$. 
\end{lemma}
\begin{proof}
Up to a translation, we may assume $p=0$. Using the results of Subsection \ref{subsec:AbelVar} and \cite[Corollary 2.3.7]{BL04}, it follows that $(-1)_A^*$ acts trivially on $\mathrm{NS}(A)$, hence 
 $D$ and $D':=(-1)_A^*D$ are two algebraically equivalent, irreducible divisors, both having a sextuple point at $0$. If $D$ and $D'$ were distinct, we would have
$D D' \geq 36$, a contradiction because $D^2=32$; thus $D=D'$. 
\end{proof}

\begin{proposition} \label{prop:curve-mult-2}
If $D \subset A$ is as in $\mathrm{Lemma}$ $\mathrm{\ref{symmetric}}$, then there exists a smooth genus $2$ curve $C$ such that $A=J(C)$. Furthermore, up to translations, the curve $D$ can be obtained as in $\mathrm{Remark}$ $\mathrm{\ref{rem:curve-6-ple}}$, namely as the image of a symmetric theta divisor $\Theta \subset A$ via the multiplication map $2_A \colon A \to A$. 
\end{proposition}
\begin{proof}
By Lemma \ref{symmetric}, we can assume that $D$ is a symmetric divisor and that its sextuple point is the origin $0 \in A$. The geometric genus of $D$ is $2$, hence its normalization $C \to D$ is a smooth genus $2$ curve.
By the universal property of the Jacobian, the composed map $C \to D \hookrightarrow A$ factors through an isogeny 
\begin{equation*}
\eta \colon J(C)\to A,
\end{equation*}
where we can assume, up to translations, that the image $\Theta$ of the embedding $C \hookrightarrow J(C)$ is a theta divisor containing the origin $0 \in J(C)$. Thus, the abelian surface $A$ is isomorphic to $J(C)/T$, where $T:= \ker \eta$ is a torsion subgroup whose order $|T|$ equals the degree $d$ of $\eta$. 
The group  $T$ contains the group generated by the six points 
\begin{equation*}
0=p_0, \, p_1, \ldots, p_5
\end{equation*}
corresponding to the six distinct points of $C$ over $0 \in D$. 
The restriction of $\eta$ to $C$ is birational,
so we have
\begin{equation*}
\eta^{*}D=\Theta_{0}+\dots+\Theta_{d-1},
\end{equation*}
where $\Theta_{0}= \Theta $ and the $\Theta_{j}$ are translates of $\Theta_0$  by the elements of $T$.
Since $D^{2}=32$, we obtain $(\eta^{*}D)^{2}=32d$. On the other hand, all the curves $\Theta_j$ are algebraically equivalent, hence $\Theta_i \Theta_j=2$ for all pairs $(i, \, j)$ and we infer $(\eta^{*}D)^{2} = (\Theta_0 + \cdots + \Theta_{d-1})^2=2d^2$. So $32d=2d^2$, that is $d=16$. 

This shows that the reducible curve $\eta^*D$ has sixteen sextuple points  $p_0, \ldots, p_{15}$, such that every curve $\Theta_j$ contains six of them; conversely, since all the $\Theta_j$ are smooth, through any of the $p_k$ pass exactly six curves. We express these facts by saying that the sixteen curves $\Theta_j$ and the sixteen points $p_k$ form a \emph{symmetric} $(16_6)$-\emph{configuration}.   
The involution $(-1)_A$ acts on $D$, so the involution $(-1)_{J(C)}$ acts on $\Theta$, that is $\Theta$ is a symmetric divisor on $J(C)$. Furthermore, the action of $(-1)_A$ induces the multiplication by $-1$ on the tangent space $T_{A, 0}$, hence it preserves the six tangent directions of $D$ at $0$; this means that $p_0, \ldots, p_{5}$ are fixed points for the restriction of $(-1)_{J(C)}$ to $\Theta$. But a non-trivial involution with six fixed points on a smooth curve of genus $2$ must be the hyperelliptic involution, so $p_0, \ldots, p_{5}$ are the Weierstrass points of $\Theta$. By \cite[Chapter 3.2, pp. 28–-39]{Mu84}, these six points generate the (order $16$) subgroup $J(C)[2]$ of points of order $2$ in $J(C)$, thus $T=J(C)[2]$.

Summing up, our symmetric $(16_6)$-configuration coincides with the so-called \emph{Kummer configuration}, see 
\cite[Chapter 10]{BL04};
 moreover, $A$ is isomorphic to $J(C)$ and the map $\eta$ coincides with the multiplication map $2_A \colon A \to A$.   
\end{proof}

\begin{theorem} \label{thm:class-type I}
Surfaces of type $I$ are precisely the product-quotient surfaces described in Section $\ref{2sec:type I}$, in particular they form a family of dimension $3$. More precisely, denoting by $\mathcal{M}_I$ their Gieseker moduli space and by $\mathcal{M}_2$ the moduli space of curves of genus $2$, there exists a surjective,  quasi-finite morphism $\mathcal{M}_I \to \mathcal{M}_2$ of degree $15$.
\end{theorem}
\begin{proof}
Given any surface $S$ of type $I$, by Proposition \ref{prop:curve-mult-2} there exists a smooth curve $C$ of genus $2$ such that $S$ is the canonical desingularization of the double cover $\alpha \colon X \to A$,  where $A=J(C)$, branched over the singular curve $D_A$ described in the example of Section $\ref{2sec:type I}$ and in Remark \ref{rem:curve-6-ple}. Equivalently, $S$ arises as a double cover $\beta \colon S \to B$, where $B= \mathrm{Sym}^2(C)$, branched over the smooth diagonal divisor $D_B$. There are sixteen distinct covers, corresponding to the sixteen square roots of $D_B$ in $\mathrm{Pic}(B)$. One of them is the double cover $u \colon C \times C \to B$, whereas the others are fifteen surfaces $S$ with $p_g(S)=q(S)=2$ and Albanese variety isomorphic to $J(C)$. We claim that, for a general choice of $C$, such surfaces are pairwise non-isomorphic. In fact, let us consider two of them, say $S_i$ and $S_j$; then, if $S_i \stackrel{\simeq}{\longrightarrow} S_j$ is an isomorphism, by the universal property of the Albanese map there exists an automorphism of abelian varieties $J(C) \stackrel{\simeq}{\longrightarrow} J(C)$ that makes the following diagram commutative:  
\begin{equation*} 
\begin{CD}
S_i  @>{\simeq} >> S_j\\
@V{}VV  @VV { } V\\
J(C) @> {\simeq}>> J(C).\\
\end{CD}
\end{equation*}
If $C$ is general, the only automorphism of $C$ is the hyperelliptic involution, so the only automorphism of $J(C)$ is the multiplication by $(-1)$, that acts trivially on the $2$-torsion divisors of $J(C)$. Consequently, the induced involution on $B$ acts trivially on the sixteen square roots of $D_B$, that is $S_i=S_j$, as claimed.

On the other hand, once fixed a curve $C$ of genus $2$, the product-quotient construction uniquely depends on the choice of the \'etale double cover $C' \to C$, that is on the choice of a non-trivial $2$-torsion element of $J(C)$. There are precisely fifteen such elements, that necessarily correspond to the fifteen surfaces with $p_g(S)=q(S)=2$ and $\textrm{Alb}(S) \simeq J(C)$ found above. 

Therefore every surface of type $I$ is a product-quotient example, and the map $\mathcal{M}_I \to \mathcal{M}_2$ defined by $[S] \mapsto [C]$ is a quasi-finite morphism of degree $15$.
\end{proof}

\begin{remark}
The moduli space of genus $2$ curves $C$ with a non-trivial $2$ torsion point in $J(C)$ is rational (see \cite{Do08}). According to the description of $\mathcal{M}_I$ in the proof of Theorem \ref{thm:class-type I}, we see that $\mathcal{M}_I$ is rational.
\end{remark}

Theorem \ref{thm:class-type I} in particular implies that the universal cover of $S$ coincides with the universal cover of $C' \times C'$, so we obtain  
 \begin{corollary}  \label{cor:univ-I}
Let $S$ be a surface of type $I$ and $\widetilde{S} \to S$ its universal cover. Then $\widetilde{S}$ is biholomorphic to the bidisk $\mathbb{H} \times \mathbb{H}$, where   $\mathbb{H} = \{z \in \mathbb{C} \; | \; \mathrm{Im} \, z >0 \}$ is the Poincar\'e upper half-plane.  
\end{corollary}

\section{Surfaces of type $II$: construction} \label{sec:type II}

\subsection{Line bundles on abelian varieties and the Appell-Humbert theorem} \label{subsec:AbelVar}

In this subsection we shortly collect some results on abelian varieties that will be used in the sequel, referring the reader to \cite[Chapters 1-4]{BL04} for more details.
Let $A=V/\Lambda$ be an abelian variety, where $V$ is a finite-dimensional $\CC$-vector space and  $\Lambda \subset V$ a lattice. Then the \emph{Appell-Humbert Theorem}, see \cite[Theorem 2.2.3]{BL04}, implies that
\begin{itemize}
\item the N\'eron-Severi group $\mathrm{NS}(A)$ can be identified with the group of hermitian forms $h \colon V \times V \to \mathbb{C}$ whose imaginary part $\mathrm{Im}\, h$ takes integral values on $\Lambda$; 
\item the Picard group $\mathrm{Pic}(A)$ can be identified with the group of pairs $(h, \, \chi)$, where $h \in \mathrm{NS}(A)$ and $\chi$ is a \emph{semicharacter}, namely a map
\begin{equation*}
\chi \colon \Lambda \to U(1), \quad \textrm{where } U(1)= \{z \in \mathbb{C} \; | \; |z|=1 \},
\end{equation*} 
 such that 
\begin{equation}\label{eq:semichar formula}
\chi(\l+\mu)=\chi(\l)\chi(\mu)e^{ \pi i \, \mathrm{Im} \, h(\l, \, \mu)} \quad \textrm{for all } \l, \, \mu \in \Lambda.
\end{equation}
\item with these identifications, the first Chern class map $c_1 \colon \mathrm{Pic}(A)\to \mathrm{NS}(A)$ is nothing but the projection to the first component, i.e. $(h, \, \chi) \mapsto h$.
\end{itemize}
We will write $\mathcal{L}= \mathcal{L}(h, \, \chi)$, so that we have $\mathcal{L}(h, \, \chi)\otimes \mathcal{L}(h', \, \chi')= \mathcal{L}(h+h', \, \chi\chi')$. The line bundle $\mathcal{L}(h, \, \chi)$ is symmetric if and only if the semicharacter $\chi$ has values in $\{\pm1\}$, see \cite[Corollary 2.3.7]{BL04}.
Furthermore, for any $\bar{v}\in A$ with
representative $v\in V$, we have 
\begin{equation} \label{eq:translation-A}
t_{\bar{v}}^{*}\LL(h,\, \chi)=\LL(h, \, \chi \, e^{2\pi i\,\mathrm{Im}\,h(v, \, \cdot)}),
\end{equation}
see \cite[Lemma 2.3.2]{BL04}.

\begin{remark} \label{rem:twice}
Assume that the class of $\mathcal{L}=\mathcal{L}(h, \, \chi)$ is $2$-divisible in 
$\mathrm{NS}(A)$, that is $h=2 h'$. Then $\mathrm{Im}\, h(	\Lambda,\, \Lambda) \subseteq 2 \mathbb{Z}$ 
and moreover formula (\ref{eq:semichar formula}) implies that $\chi  \colon \Lambda \to U(1)$ is a character, namely $\chi(\l + \mu) = \chi(\lambda) \chi(\mu)$. In particular, $\mathcal{L}$ belongs to $\mathrm{Pic}^0(A)$ if and only if there exists a character $\chi$ such that $\LL = \LL(0, \, \chi)$. 
\end{remark}

\begin{proposition}[\cite{BL04}, Lemma 2.3.4] \label{Prop:BirLangeEssential} 
Let $A_1 = V_1/ \Lambda_1$ and $A_2 = V_2/ \Lambda_2$ be two abelian varieties, and let $f \colon A_2 \to A_1$ be a homomorphism with
analytic representation $F \colon V_2 \to V_1$ and rational representation $F_{\Lambda} \colon \Lambda_2 \to \Lambda_1$. Then for any $\mathcal{L}(h, \, \chi) \in \mathrm{Pic}(A_1)$ we have
\begin{equation}  \label{eq:L-pullback}
f^{*}\LL(h, \, \chi)=\LL(F^{*}h, \, F_{\Lambda}^{*}\chi),
\end{equation}
\end{proposition}

Given a point $x \in A$ and a divisor $D \subset A$, let us denote by $m(D, \, x)$ the multiplicity
of $D$ at $x$. 
\begin{lemma}[\cite{BL04}, Proposition 4.7.2] \label{lem:SymDiv} 
Let $\LL=\LL(h, \, \chi)$ be a symmetric line bundle on $A$ and $D$ a symmetric effective
divisor such that $\LL = \mathcal{O}_A(D)$. For every $2$-torsion point $x \in A[2]$ with representative $\frac{1}{2} \lambda$, where $\lambda \in \Lambda$, we have
\begin{equation*}
\chi(\l)=(-1)^{m(D, \, 0)+m(D, \, x)}.
\end{equation*}
\end{lemma}

\subsection{The equianharmonic product} \label{subsec:Hirzebruch}

Let $\zeta:=e^{2 \pi i/6}= \frac{1}{2} + \frac{\sqrt{3}}{2} i$, so that $\zeta^2-\zeta+1=0$, and consider the \emph{equianharmonic elliptic curve}
\begin{equation} \label{eq:curve-E'}
E':=\mathbb{C}/ \Gamma_{\zeta}, \quad \Gamma_{\zeta}:= \mathbb{Z} \zeta \oplus \mathbb{Z}. 
\end{equation}
Setting $V:= \mathbb{C}^2$, we can define
\begin{equation*}
A' := E' \times E' = V/\Lambda_{A'}, \quad \Lambda_{A'}: = \Gamma_{\zeta} \times \Gamma_{\zeta}.
\end{equation*}   
Then $A'$ is a principally polarized abelian surface, that we will call the \emph{equianharmonic product}.
Denoting by $(z_1, \, z_2)$ the coordinates of $V$ 
and by $e_1=(1, \, 0)$, $e_2 = (0, \, 1)$ its standard basis, the four vectors
\begin{equation} \label{eq:basis-l-m}
\lambda_1 := \zeta e_1,\ \ \lambda_2 := \zeta e_2,\ \ e_1,\ \ e_2 
\end{equation}
form a basis for the lattice $\Lambda_{A'}$.

We now consider the four $1$-dimensional complex subspaces of $V$ defined as
\begin{equation} \label{eq:four-complex-lines}
\begin{aligned}
V_1 & := \textrm{span}(e_1) = \{z_2=0\},   \quad \quad \quad \quad \,  V_2 := \textrm{span}(e_2) = \{z_1=0\}, \\
V_3 & := \textrm{span}(e_1+e_2) = \{z_1-z_2=0\},   V_4 := \textrm{span}(e_1 +\zeta e_2) =\{\zeta z_1 - z_2 =0 \}. 
\end{aligned}
\end{equation}
For each $k \in \{1, \, 2, \, 3, \, 4\}$, the subspace $V_k$ contains a rank $2$ sublattice $\Lambda_k \subset \Lambda_{A'}$ isomorphic to $\Gamma_{\zeta}$, where
\begin{equation} \label{eq:four-sublattices}
\begin{aligned}
\Lambda_1 & := \mathbb{Z} \lambda_1 \oplus \mathbb{Z} e_1,  \quad \quad \quad \quad \quad \quad
\Lambda_2 := \mathbb{Z} \lambda_2 \oplus \mathbb{Z} e_2, \\
\Lambda_3 & := \mathbb{Z}(\lambda_1 + \lambda_2) \oplus \mathbb{Z}(e_1+ e_2), 
\Lambda_4 :=\mathbb{Z}(\lambda_1 + \lambda_2 - e_2) \oplus \mathbb{Z}(\lambda_2 + e_1).
\end{aligned}
\end{equation}

Consequently, in $A'$ there are four elliptic curves isomorphic to $E'$, namely
\begin{equation}
E'_k := V_k/ \Lambda_k, \quad k \in \{1, \, 2, \, 3, \, 4 \}.
\end{equation}
\begin{proposition}[\cite{Hir84} Section 1] \label{prop:Hir-quadruple-point}
The four curves $E'_k$ only intersect $($pairwise transversally$)$ at the origin $o' \in A'$. Consequently, the reducible divisor
\begin{equation*}
D_{A'} := E'_1+E'_2+E'_3+E'_4
\end{equation*}
has an ordinary quadruple point at $o'$ and no other singularities. 
\end{proposition}

By the Appell-Humbert Theorem, the N\'eron-Severi group  $\textrm{NS}(A')$ of $A'$ 
can be identified with the group of hermitian forms $h$ on $V$ whose imaginary part takes integral values on $\Lambda_{A'}$. We will use the  symbol $H$ for the $2 \times 2$ hermitian matrix associated to $h$ with respect to the standard basis of $V$ so that, thinking of $v, \, w \in V$ as column vectors, we can write
$h(v, \, w) = {}^tv H \bar{w}.$
We want now to identify those hermitian matrices $H_1, \ldots, H_4$
that correspond to the classes of the curves $E_1', \ldots, E_4'$, respectively. 
\begin{proposition} \label{pro:class_DA'} 
We have 
\begin{equation*}
\begin{split}H_{1} & =\frac{2}{\sqrt{3}}\left(\begin{array}{cc}
0 & 0\\
0 & 1
\end{array}\right),\quad\quad\;\,H_{2}=\frac{2}{\sqrt{3}}\left(\begin{array}{cc}
1 & 0\\
0 & 0
\end{array}\right),\\
H_{3} & =\frac{2}{\sqrt{3}}\left(\begin{array}{cc}
\;\;1 & -1\\
-1 & \;\;1
\end{array}\right),\quad H_{4}=\frac{2}{\sqrt{3}}\left(\begin{array}{cc}
\;1 & -\zeta\\
-\bar{\zeta} & \;1
\end{array}\right),
\end{split}
\end{equation*}
so that the hermitian matrix representing in $\mathrm{NS}(A')$ the class  of the divisor $D_{A'}$ is
\begin{equation*}
H:=H_1+H_2+H_3+H_4 = \frac{2}{\sqrt{3}} \left(\begin{array}{cc}
3 & -1-\zeta\\
-1- \bar{\zeta} & 3
\end{array}\right).
\end{equation*}
Moreover, setting $\lambda = (a_1+\zeta a_2, \, a_3+\zeta a_4) \in \Lambda_{A'}$, the semicharacter $\chi_{D_{A'}}$ corresponding to the line bundle  $\OO_{A'}(D_{A'})$ can be written as 
\begin{equation*}
\chi_{D_{A'}}(\lambda)=(-1)^{a_{1}+a_{2}+a_{3}+a_{4}+a_1 (a_2+ a_3+a_4)+(a_2 +a_3)a_4}.
\end{equation*}
\end{proposition}
\begin{proof}
The hermitian form $\tilde{h}$ on $\mathbb{C}$ given by $\tilde{h}(z_1, \, z_2)=\frac{2}{\sqrt{3}}z_1\bar{z_2}$ is positive
definite and its imaginary part is integer-valued on $\Gamma_{\zeta}$, so it defines a positive class in $\mathrm{NS}(E')$. Moreover, in the ordered basis $\{\zeta, \, 1\}$ of $\Gamma_{\zeta}$, the alternating form $\mathrm{Im}\, \tilde{h}$ is represented by the skew-symmetric matrix  $\begin{psmallmatrix}\;\;0 & 1 \\ -1 & 0\end{psmallmatrix}$,
whose Pfaffian equals $1$, so $\tilde{h}$ corresponds to the ample generator
of the N\'eron-Severi group of $E'$, see \cite[Corollary 3.2.8]{BL04}. In other words, $\tilde{h}$ is the Chern class of $\OO_{E'}(0)$,
where $0$ is the origin of $E'$. Write $\OO_{E'}(0)=\mathcal{L}(\tilde{h}, \, \nu)$ 
for a suitable semicharacter $\nu \colon \Gamma_{\zeta} \to \mathbb{C}$;
since $\OO_{E'}(0)$
is a symmetric line bundle, the values of $\nu$ at
the generators of $\Gamma_{\zeta}$ can be computed by using Lemma \ref{lem:SymDiv}, obtaining 
$\nu(1)=-1, \quad \nu(\zeta)=-1.$
Consequently, for all $a, \, b \in \mathbb{Z}$ we get 
\begin{equation} \label{eq:char-psi}
\begin{split}
\nu(a+b\zeta)&=\nu(a) \nu(b \zeta) e^{\pi i \, \mathrm{Im} \, \tilde{h}(a, \, b \zeta)}=(-1)^a (-1)^b (-1)^{ab} = (-1)^{a+b+ab}.
\end{split}
\end{equation}
For any $k \in \{1,\dots, 4\}$ let us define a group homomorphism $F_{k} \colon A'\to E'$ as follows: 
\begin{equation*}
F_{1}(z_{1}, \, z_{2})=z_{2}, \quad F_{2}(z_{1}, \, z_{2})=z_{1},\quad  F_{3}(z_{1}, \, z_{2})=z_{1}-z_{2},\quad F_{4}(z_{1}, \, z_{2})=\zeta z_{1}-z_{2}.
\end{equation*}
By \eqref{eq:four-complex-lines} we have $E_{k}'=F_k^{*}(0)$ and so, setting $\mathcal{O}_{A'}(E_k')=\mathcal{L}(h_k, \, \chi_k')$, by \eqref{eq:L-pullback} we deduce  
\begin{equation} \label{eq:Fk*}
h_k=F_k^* \tilde{h}, \quad \chi_k'=F_k^* \nu.
\end{equation}
This gives immediately the four matrices $H_{1},\dots,H_{4}$. Moreover, by using  \eqref{eq:char-psi} and \eqref{eq:Fk*}, we can write down the semicharacters $\chi'_{1},\dots,\chi'_{4}$; in fact, for any $\lambda = (a_1+\zeta a_2, \, a_3+\zeta a_4) \in \Lambda_{A'}$, we obtain
\begin{equation*}
\begin{split}
\chi'_{1}(\lambda)&=(-1)^{a_{3}+a_{4}+a_{3}a_{4}}\\
\chi'_{2}(\lambda)&=(-1)^{a_{1}+a_{2}+a_{1}a_{2}}\\
\chi'_{3}(\lambda)&=(-1)^{a_{1}+a_{2}+a_{3}+a_{4}+(a_{1}+a_{3})(a_{2}+a_{4})}\\
\chi'_{4}(\lambda)&=(-1)^{a_{1}+a_{2}+a_{3}+a_{4}+(a_{1}+a_{4})(a_{2}+a_{3})+a_{2}a_{3}}.
\end{split}
\end{equation*}
The semicharacter $\chi_{D_{A'}}$ can be now computed by using the formula 
$\chi_{D_{A'}}=\chi'_{1}\chi'_{2} \chi'_{3}\chi'_{4}.$
\end{proof}

\begin{remark} \label{rem:princ-pol}
The hermitian matrix
\begin{equation*}
H_1+H_2= \frac{2}{\sqrt{3}} \left(\begin{array}{cc}
1 & 0\\
0 & 1
\end{array}\right)
\end{equation*}
represents  in $\mathrm{NS}(A')$ the class of the principal polarization of product type
\begin{equation*}
\Theta := E' \times \{0\} + \{0\} \times E'.
\end{equation*}
\end{remark}

\begin{remark} \label{prop:neron-Severi-via-E}
The free abelian group $\mathrm{NS}(A')$ is generated by the classes of the elliptic curves $E_1'$, $E_2'$, $E_3'$, $E_4'$. In fact, since $A' = E' \times E'$ and $E'$ has complex multiplication, it is well-known that $\textrm{NS}(A')$ has rank $4$, see \cite[Exercise 5.6 (10) p.142]{BL04}, hence we only need to show that the classes of the curves $E_k'$ generate a primitive sublattice of maximal rank in the N\'eron-Severi group.  By Proposition \ref{prop:Hir-quadruple-point}, the corresponding Gram matrix has determinant
\begin{equation*}
\det \, (E_i' \cdot E_j') = \det(1-\delta_{ij})= -3 
\end{equation*} 
so the claim follows because $-3$ is a non-zero, square-free integer.
\end{remark}

\subsection{Double covers of the equianharmonic product} \label{subsec:example-II}

In order to construct a surface of type $II$, we must find an abelian surface $A$ and a divisor $D_A$ on it such that  
\begin{itemize}
\item $D_A$ is $2$-divisible in $\textrm{Pic}(A)$; 
\item $D_A^2=24$ and $D_A$ has precisely two ordinary quadruple points as singularities.
\end{itemize}
We will construct the pair $(A, \, D_A)$ as an \'etale double cover of the pair $(A', \, D_{A'})$, where $A'=V/\Lambda_{A'}$ is the equianharmonic product and $D_{A'}=E_1'+E_2'+E_3'+E_4'$ is the sum of four elliptic curves considered in Proposition \ref{prop:Hir-quadruple-point}.

By the Appell-Humbert theorem, the sixteen $2$-torsion divisors on $A'$, i.e. the elements of order $2$ in $\textrm{Pic}^0(A')$, correspond to the sixteen characters 
\begin{equation} \label{eq:char-chi}
\chi \colon \Lambda_{A'} \to \{ \pm 1 \}.
\end{equation}
Any such character is specified 
by its values at the elements of the ordered basis $\{\lambda_1, \, \lambda_2, \, e_1, \, e_2 \}$ of $\Lambda_{A'}$ given in \eqref{eq:basis-l-m}, so it can be denoted by 
\begin{equation*}
\chi = (\chi(\lambda_1), \, \chi(\lambda_2), \, \chi(e_1), \, \chi(e_2)). 
\end{equation*}
For instance, $\chi_0:=(1, \, 1,\, 1,\, 1)$ is the trivial character, corresponding to the trivial divisor $\mathcal{O}_{A'}$.
We will write
\begin{equation} \label{eq:characters}
\begin{array}{lll}
\chi_1\ :=(-1, \, -1, \, 1, -1),    & \chi_2\ := (1, \, -1, \, -1, \, 1),  & \chi_3\ :=(-1, \, 1, \, -1, \, -1),\\ 
\chi_4\ :=(1, \, 1, \, -1, \, 1),   &  \chi_5\ :=(-1, \, 1, \, 1, \,1),    & \chi_6\ :=(-1, \, 1,\, -1, \, 1),\\ 
\chi_7\ :=(1, \, 1, \, 1, \, -1),   &  \chi_8\ :=(1, \, -1, \, 1, \, -1),  & \chi_9\ :=(-1, \, 1, \, 1, \,-1),\\
\chi_{10} :=(1, \, 1,\, -1, \, -1), & \chi_{11}:=(-1, \, -1, \, -1, \, 1), & \chi_{12} :=(1, \, -1, \, 1, \, 1),\\
\chi_{13}:=(1, \, -1, \, -1, \,-1), &  \chi_{14}:=(-1, \, -1,\, 1, \, 1),  &  \chi_{15} :=(-1, \, -1, \, -1, \, -1)
\end{array}
\end{equation} 
for the fifteen non-trivial characters. To any non-trivial $2$-torsion divisor on $A'$, and so to any non-trivial character $\chi$ as in \eqref{eq:char-chi}, it corresponds an isogeny of degree two 
$f_{\chi} \colon A_{\chi} \to A';$
in fact, $\ker  \chi \subset \Lambda_{A'}$ is a sublattice of index $2$ and $A_{\chi}$ is the abelian surface  
\begin{equation} \label{eq:A-chi}
A_{\chi} = V/\ker  \chi.
\end{equation}
Let us set 
\begin{equation*}
E_i  := f_{\chi}^*(E_i'), \quad
D_{A_{\chi}}  := f_{\chi}^*(D_{A'})=E_1+E_2+E_3+E_4
\end{equation*}
and write $\varSigma$ for the subgroup  of $\textrm{Pic}^0(A')$ generated by $\chi_1$ and $\chi_2$, namely
\begin{equation} \label{Characters}
\varSigma:= \{\chi_0, \, \chi_1, \, \chi_2, \, \chi_3  \}.
\end{equation}
We are now ready to prove the key result of this subsection.
\begin{proposition} \label{prop:2-divisibility}
The following are equivalent$:$ 
\begin{itemize}
\item[$\boldsymbol{(a)}$] the divisor $D_{A_{\chi}}$ is $2$-divisible in $\mathrm{Pic}(A_{\chi});$ 
\item[$\boldsymbol{(a')}$] the divisor $D_{A_{\chi}}$ is $2$-divisible in $\mathrm{NS}(A_{\chi});$ 
\item[$\boldsymbol{(b)}$] every $E_i$ is an irreducible elliptic curve in $A_{\chi};$
\item[$\boldsymbol{(c)}$] the character $\chi$ is a non-trivial element of $\varSigma$.  
\end{itemize}
\end{proposition}
\begin{proof}
We first observe that $\mathrm{NS}(A_{\chi})=\mathrm{Pic}(A_{\chi})/\mathrm{Pic}^0(A_{\chi})$ and $\mathrm{Pic}^0(A_{\chi})$ is a divisible group, so $\boldsymbol{(a)}$ is equivalent to $\boldsymbol{(a')}$.  

Next, the curve $E_i \subset A_{\chi}$ is irreducible if and only if the $2$-torsion divisor corresponding to the character  $\chi \colon \Lambda_{A'} \to \{\pm1\}$ restricts non-trivially to $E_i'$. This in turn means that $\chi$  restricts non-trivially to the sublattice $\Lambda_i$, and so $\boldsymbol{(b)}$ occurs if and only if 
$\chi$ restricts non-trivially to all $\Lambda_1$, $\Lambda_2$, 
$\Lambda_3$, $\Lambda_4$. By using the generators given in \eqref{eq:four-sublattices}, a long but elementary computation (or a quick computer calculation) shows that this happens if and only if $\boldsymbol{(c)}$ holds.  

It remains to prove that $\boldsymbol{(a')}$ and $\boldsymbol{(c)}$ are equivalent. The isogeny $f_{\chi} \colon A_{\chi} \to A$ lifts to the identity $1_{V} \colon V \to V$  so, if $h \colon V \times V \to \mathbb{C}$ is the hermitian form that represents the class of $D_{A'}$ in $\mathrm{NS}(A')$, then the same form also represents the class of $D_{A_{\chi}}$ in $\mathrm{NS}(A_{\chi})$.  
By the Appell-Humbert theorem the group $\NS(A_{\chi})$ can be identified with the group of hermitian forms on $V$ whose imaginary part takes integral values on the lattice $\ker \chi$, so \eqref{eq:A-chi} implies that condition $\boldsymbol{(a')}$ is equivalent to
\begin{equation} \label{eq:(a')}
\mathrm{Im}\; h(\ker \chi, \, \ker \chi) \subseteq 2 \mathbb{Z}.
\end{equation} 
The non-zero values assumed by the alternating form $\textrm{Im} \, h$ on the generators $\lambda_1, \, \lambda_2, \, e_1, \, e_2$ of $\Lambda_{A'}$ can be computed by using the hermitian matrix $H$ given in Proposition \ref{pro:class_DA'}, obtaining Table \ref{table:imh} below:
\begin{table}[H] 
\begin{center}
\begin{tabular}{c|c|c|c|c|c|c} 
$(\cdot, \, \cdot)$   & $(\lambda_1, \, \lambda_2)$ & $(\lambda_1, \, e_1)$ & $(\lambda_1, \, e_2)$ & $(\lambda_2, \, e_1)$ & $(\lambda_2, \, e_2) $ & $(e_1, \, e_2)$ \\
\hline
$\mathrm{Im} \; h(\cdot, \, \cdot)$ & $-1$ & $3$ & $-2$ & $-1$ & $3$ & $-1$ \\
\end{tabular} \caption{Non-zero values of $\imh$ at the generators of $\Lambda_{A'}$} \label{table:imh}  
\end{center}
\end{table}
Now we show that (\ref{eq:(a')}) holds if and only if  $\chi$ is a non-trivial element of $\varSigma.$ In fact we have seen that, if $\chi \notin \{\chi_1, \, \chi_2, \, \chi_3\}$, then
one of the effective divisors $E_i=f_{\chi}^*(E_i')$ is a disjoint union of two elliptic curves, say $E_i=E_{i1}+E_{i2}.$ But then, using the projection formula, we find
\begin{equation*}
D_{A_{\chi}} \cdot E_{i1}  = f_{\chi}^*(D_{A'}) \cdot E_{i1} = D_{A'} \cdot f_{\chi \,*}(E_{i1})  = D_{A'} \cdot E_i' =3
\end{equation*}
which is not an even integer, so $D_{A_{\chi}}$ is not $2$-divisible in this situation.

Let us consider now the case $\chi \in \{\chi_1, \, \chi_2, \, \chi_3 \}$. We can easily see that the integral bases of $\ker \chi_1$, $\ker \chi_2$, $\ker \chi_3$  are given by
\begin{equation} \label{eq:bases-ker-chi}
\begin{split}
\mathscr{B}_1&:=\{e_1, \, \lambda_1+e_2, \, \lambda_2+e_2, \, 2e_2\}, \\
\mathscr{B}_2&:=\{\lambda_2+e_1, \, \lambda_1, \, e_2, \, 2e_1 \},\\
\mathscr{B}_3&:=\{\lambda_1+e_2, \, \lambda_2, \, 2e_2, \, e_1+e_2 \},\\
\end{split}
\end{equation}
respectively. Then, by using Table \ref{table:imh}, it is straightforward to  check that 
$\mathrm{Im}\, h(b_1, \, b_2) \in 2 \mathbb{Z}$ for all $b_1, \,  b_2\in \mathscr{B}_1$;
for instance, we have
\begin{equation*} \label{eq1}
\begin{split}
\mathrm{Im}\; h(\lambda_1+e_2, \, \lambda_2+e_2) & = \mathrm{Im}\; h(\lambda_1,\lambda_2) + \mathrm{Im}\; h(\lambda_1,e_2) + \mathrm{Im}\; h(e_2,\lambda_2) + \mathrm{Im}\; h(e_2,e_2) \\
 & = -1 -2 -3 +0 = -6 \in 2 \mathbb{Z}.
\end{split}
\end{equation*}
This shows that the inclusion (\ref{eq:(a')}) holds for $\chi_1.$ The proof that it also holds for $\chi_2$ and $\chi_3$ is analogous.
\end{proof}

\begin{remark} \label{rem:duality-for-2-torsion}
Writing the details in the proof of Proposition \ref{prop:2-divisibility}, one sees that every non-trivial character $\chi$ in \eqref{eq:characters} restricts trivially to \emph{at most one} curve $E_i'$.  Identifying $A'$ with $\textrm{Pic}^0(A')$ via the principal polarization $\Theta$ described in Remark \ref{rem:princ-pol}, this corresponds to the fact that every non-zero $2$-torsion point of $A'$ is contained in at most one of the $E_i'$. More precisely, every $E_i'$ contains exactly three non-zero $2$-torsion points of $A'$, so it remain $16-(4\times 3 + 1)=3$ of them that are not contained in any of the $E_i'$. Via the identification above, they clearly correspond to the three $2$-torsion divisors restricting non-trivially to all the $E_i'$, namely to the three non-trivial characters in the group $\varSigma$.    
\end{remark}

Summing up, we have the following existence result for surfaces of type $II$.
\begin{proposition} \label{prop:type-II}
Let $\chi$ be any non-trivial element in the group $\varSigma$, and write $f \colon A \to A'$ instead of $f_{\chi} \colon A_{\chi} \to A'$. Then there exists a double cover
$\alpha_X \colon X \to A $
branched precisely over the $2$-divisible effective divisor $D_A \subset A$. The minimal resolution $S$ of $X$ is a smooth surface with $p_g=q=2$, $K^2=8$ and Albanese map of degree $2$, belonging to type $II$. 

\end{proposition}
\begin{proof}
It only remains to compute the invariants of $S$. From the double cover formulas (see \cite[Chapter V.22]{BHPV03}) we see that, if we impose an ordinary quadruple point to the branch locus, then $\chi$ decreases  by $1$ and $K^2$ decreases by $2$, hence we get 
\begin{equation*}
\chi(\mathcal O_S)=\frac{1}{8}D_A^2-2=1 {\rm\ \ \ and\ \ \ } K_S^2=\frac{1}{2}D_A^2-4=8.
\end{equation*}
Since $q(S)\geq q(A)=2,$ we have $p_g(S)=q(S)\geq 2.$
Assume that $p_g(S)=q(S)\geq 3.$ By \cite{HacPar}, \cite{Piro} and \cite[Beauville appendix]{Deb81}, we have two possibilities:
\begin{itemize}
\item $p_g(S)=q(S)=4$ and $S$ is the product of two curves of genus $2$;
\item $p_g(S)=q(S)=3$ and $S=(C_2 \times C_3)/\mathbb{Z}_2$, where $C_2$ is a smooth curve of genus $2$ with an elliptic involution $\tau_2$, $C_3$ is a smooth curve of genus $3$ with a free involution $\tau_3$, and the cyclic group $\mathbb{Z}_2$ acts freely on the product $C_2 \times C_3$ via the involution $\tau_2 \times \tau_3$. 
\end{itemize}
In both cases above,  $S$ contains no elliptic curves. On the other hand, all our surfaces of type $II$  contain four elliptic curves, coming from the strict transform of $D_A$.  Therefore the only possibility is $p_g(S)=q(S)=2$.
\end{proof}

\section{Surfaces of type $II$: classification}
\label{sec:classification-II}


\subsection{Holomorphic and anti-holomorphic diffeomorphisms of cyclic covers}

In this Section we discuss about lifts on cyclic covers of automorphisms or anti-automorphisms. We use methods and results of Pardini, see \cite{Pa91}.

Let $n \geq 2$ be an integer and let $D$ be an effective divisor
on a smooth projective variety $Y$,
such that
\begin{equation*}
\OO_{Y}(D) \simeq \mathcal{L}_{1}^{\otimes n} \simeq \mathcal{L}_{2}^{\otimes n}
\end{equation*}
for some line bundles $\mathcal{L}_{1},\,\mathcal{L}_{2} \in \mathrm{Pic}(Y)$.
Canonically associated to such data, there exists two simple $n$-cyclic covers
\begin{equation*}
\pi_{1} \colon X_{1}\to Y \quad \text{and} \quad \pi_{2} \colon X_{2}\to Y, 
\end{equation*}
both branched over $D$. We want to provide conditions ensuring that the two compact complex manifolds  
underlying $X_{1}$ and $X_{2}$ are biholomorphic or anti-biholomorphic. 

Following \cite[Section3]{KK}, let us denote
by $\Kl(Y)$ the group of holomorphic and anti-holomorphic diffeomorphisms of $Y$.
There is a short exact sequence
\begin{equation*}
1 \longrightarrow \mathrm{Aut}(Y) \longrightarrow \Kl(Y) \longrightarrow H \to 1,
\end{equation*}
where $H=\mathbb{Z}/2 \mathbb{Z}$ or $H=0$. 
To any anti-holomorphic element $\s \in\Kl(Y)$ we can associate a $\mathbb{C}$-antilinear map 
\begin{equation*}
\sigma^* \colon \CC(Y) \longrightarrow \CC(Y)
\end{equation*}
on the function field $\CC(Y)$ by defining
\begin{equation*}
(\sigma^*f)(x):= \overline{f(\sigma(x))}
\end{equation*}
for all $f \in \CC(Y)$. That action extends the usual action of $\aut(Y)$ on $\CC(Y)$ in a natural way (note that in \cite{KK} the notation $\sigma^*$ is used only for holomorphic maps, whereas for anti-holomorphic maps the corresponding notation is $\sigma^{!}$). 
We have $
 \sigma^{-1}( \mathrm{div \,}(f)) = \mathrm{div \,} (\sigma^*f),$ hence the action $\sigma^* \colon \mathrm{Div}(Y) \to \mathrm{Div}(Y)$ induces an action $\sigma^* \colon \mathrm{Pic}(Y) \to \mathrm{Pic}(Y)$, such that $\sigma^* K_Y = K_Y$, in the usual way. Namely, if the line bundle $\mathcal{L} \in \mathrm{Pic}(Y)$ is defined by the transition functions $\{g_{ij} \}$ with respect to the open cover $\{U_i\}$, then $\sigma^* \mathcal{L}$ is determined by the transition functions $\{\sigma^* g_{ij}\}$ with respect to the cover $\{\sigma^{-1}(U_i) \}$; furthermore, given a open set $U \subseteq Y$ and a holomorphic $r$-form $\omega = \sum g_{i_1 \ldots i_r} dx_{i_1} \wedge \ldots \wedge dx_{i_r} \in \Gamma(U, \, K_Y)$, its pullback via $\sigma$ is the holomorphic $r$-form   
 $\sigma^* \omega = \sum \sigma^*g_{i_1\ldots i_r} \, d(\sigma^*x_{i_1}) \wedge \ldots \wedge d(\sigma^*x_{i_r}) \in \Gamma(\sigma^{-1}(U), \, K_Y)$. 
 Moreover, the intersection numbers are also preserved by the action of any $\sigma \in \Kl(Y)$.  
   
 \begin{example} \label{ex:anti-holo-A}
Let $A_1 = V_1/ \Lambda_1$ and $A_2 = V_2/ \Lambda_2$ be two abelian varieties, and let $\sigma \colon A_2 \to A_1$ be an anti-holomorphic homomorphism with analytic representation $\mathfrak{S} \colon V_2 \to V_1$ and rational representation $\mathfrak{S} _{\Lambda} \colon \Lambda_2 \to \Lambda_1$ (note that $\mathfrak{S}$ is a $\mathbb{C}$-antilinear map). Then, for any $\mathcal{L}(h, \, \chi) \in \mathrm{Pic}(A_1)$, we have the following analogue of \eqref{eq:L-pullback}: 
\begin{equation} \label{eq:L-anti-holo-pullback}
\sigma^{*}\LL(h, \, \chi) = \LL(\overline{\mathfrak{S}^{*}h},\, \overline{\mathfrak{S}_{\Lambda}^{*}\chi}),
\end{equation}
In fact, looking at the transition function of the anti-holomorphic line bundle $\LL(\mathfrak{S}^{*}h, \, \mathfrak{S}_{\Lambda}^{*}\chi)$ we see that, in order to obtain a holomorphic one, we must take the conjugated hermitian form $\overline{\mathfrak{S}^{*}h}$
and the conjugated semicharacter $\overline{\mathfrak{S}_{\Lambda}^{*}\chi}$.  
\end{example} 

Let us now denote by $\Kl(Y, \, D)$ and $\aut(Y, \, D)$ the subgroups of
$\Kl(Y)$ and $\aut(Y)$ given by diffeomorphisms such that $\sigma^*D=D$. Again, $\aut(Y, \, D)$
is a normal subgroup of $\Kl(Y, \, D)$, of index $1$ or $2$. 
\begin{proposition} \label{prop:antiHolo and Holo} ${}$ 
\begin{itemize}
\item[$\boldsymbol{(i)}$] Let $\s\in\Kl(Y, \, D)$ be such that
$\s^{*}\mathcal{L}_{2} \simeq \mathcal{L}_{1}$. Then there exists a diffeomorphism 
$\tilde{\s} \colon X_{1}\to X_{2}$
such that $\sigma \circ \pi_{1}=\pi_{2}\circ\tilde{\s}$. Moreover, $\tilde{\sigma}$ is holomorphic $($respectively, anti-holomorphic$)$ if and only if $\sigma$ is so.
\item[$\boldsymbol{(ii)}$]
Let be $\sigma \in \Kl(Y, \, D)$. If $\s^* \mathcal{L}_{i}\simeq \mathcal{L}_{i}$, then $\sigma$  lifts to $X_i$ and there are $n$ different such lifts.  
\item[$\boldsymbol{(iii)}$] Let be $\tilde \s\in\Kl(X_i)$ such that $\tilde \s \circ \pi_i=\pi_i$. Then $\tilde s$ induces an automorphism $\s \in \Kl(Y,D)$. Moreover if either $D>0$ or $n=2$, one has $\s^* \mathcal{L}_{i}\simeq \mathcal{L}_{i}$. 
\end{itemize}
\end{proposition}
\begin{proof} 
Let us prove $\boldsymbol{(i)}$. Let $\mathbb{L}_1$, $\mathbb{L}_2$ be the total spaces of $\mathcal{L}_1$, $\mathcal{L}_2$ and let $p_1 \colon \mathbb{L}_1 \to Y$, $p_2\colon \mathbb{L}_2 \to Y$ be the corresponding projections. Let $s \in H^0(Y, \, \mathcal{O}_Y(D))$ be a section vanishing exactly along $D$ (if $D=0$, we take for $s$ the constant function $1$). If $t_i \in H^0(\mathbb{L}_i, \, p_i^* \mathcal{L}_i)$ denotes the tautological section, by \cite[I.17]{BHPV03} it follows that global equations for $X_1$ and $X_2$, as analytic subvarieties of $\mathbb{L}_1$ and $\mathbb{L}_2$, are provided by
\begin{equation*}
t_1^n-p_1^*s=0 \quad \mathrm{and} \quad t_2^n-p_2^*s=0,
\end{equation*} 
the covering maps $\pi_1$ and $\pi_2$ being induced by the restrictions of $p_1$ and $p_2$, respectively. Since $\sigma \in \Kl(Y, \, D)$, we have $\sigma^*s = \lambda s$ with $\lambda \in \mathbb{C}^*$.  Moreover, $\s^{*}\mathcal{L}_{2} \simeq \mathcal{L}_{1}$  implies that there exists a diffeomorphism $\tilde{\sigma} \colon \mathbb{L}_1 \to \mathbb{L}_2$ such that $p_2 \circ \tilde{\sigma}= \sigma \circ p_1$, hence
\begin{equation*}
\tilde{\sigma}^*(p_2^* s)=p_1^* (\sigma^* s)= \lambda \, p_1^* s. 
\end{equation*}
Moreover, we have $\tilde{\sigma}^* t_2 = \mu t_1$, with $\mu \in \mathbb{C}^*$. Up to rescaling $t_2$ by a constant factor we can assume $\mu = \sqrt[n]{\lambda}$, so that
\begin{equation*}
\tilde{\sigma}^*(t_2^n-p_2^*s)=\lambda(t_1^n-p_1^*s).
\end{equation*}
This means that $\tilde{\sigma}\colon \mathbb{L}_1 \to \mathbb{L}_2$ restricts to a diffeomorphism $\tilde{\sigma} \colon X_1 \to X_2$, which is compatible with the two covering maps $\pi_1$ and $\pi_2$. By construction, such a diffeomorphism is holomorphic (respectively, anti-holomorphic) if and only if $\sigma$ is so. \\
The part $\boldsymbol{(ii)}$ follows from part $\boldsymbol{(i)}$, setting $\mathcal{L}_1=\mathcal{L}_2$, so that $X_1=X_2$. Any two lifts differ by an automorphism of the cover that induces the identity on Y, thus there are $n$ different lifts of $\sigma$.\\
Let us prove part $\boldsymbol{(iii)}$. Since $\tilde \s$ preserves $\pi_i$, the induced automorphism $\sigma$ of $Y$ must preserve $D$.
Let $g$ be a generator of the Galois group $G$ of $\pi_i$.
There exists an integer $a$ prime with $n$ such that $\tilde \s$ satisfies $\tilde \s (gx)=g^a(\tilde \s x)$ for all $x \in X_i$. 
Thus the induced automorphism $\s^*$ of $\pi_{i*} \mathcal{O}_{X_i}$ permutes the summands of the decomposition under the action of $G$. 
Since the cyclic cover is simple, if $D > 0$ by looking at the Chern classes of the summands, one can see that we must have $\s^*(\mathcal{L}_i)=\mathcal{L}_i$.
 In the case of a double cover, the statement is true also in the  \'etale case, since there is only one non-trivial summand in the decomposition of $\pi_{i*}$.
\end{proof}

In the case of double covers induced by the Albanese map, we have the following converse of Proposition \ref{prop:antiHolo and Holo}, $\boldsymbol{(i)}$.

\begin{proposition} \label{prop:converse-holo}

Set $n=2$, let $Y=A$ be an abelian variety and assume that the double cover $\pi_i \colon X_i \to A$ is the Albanese map of $X_i$, for $i=1,\, 2$. If there is a holomorphic $($respectively, anti-holomorphic$)$ diffeomorphism $\tilde{\sigma} \colon X_1 \to X_2$, then there exists a holomorphic $($respectively, anti-holomorphic$)$ diffeomorphism  $\sigma \in \Kl(A, \, D)$ such that  $\sigma^{*}\mathcal{L}_{2}\simeq \mathcal{L}_{1}$. 
\end{proposition}
\begin{proof}
We first assume that $\tilde{\sigma}$ is holomorphic.
By the universal property of the Albanese map the morphism $\pi_2 \circ \tilde{\sigma} \colon X_1 \to A$ factors through $\pi_1$, in other words there exists $\sigma \colon A \to A$ such that $\sigma \circ \pi_1 = \pi_2 \circ \tilde{\sigma}$. The map $\sigma$ is an isomorphism because $\tilde{\sigma}$ is an isomorphism, then it sends the branch locus of $\pi_1$ to the branch locus of $\pi_2$, or equivalently $\sigma^*D=D$. Finally, looking at the direct image of the structural sheaf $\mathcal{O}_{X_1}$ we get
\begin{equation*}
\begin{split}
(\sigma \circ \pi_1)_* \mathcal{O}_{X_1} & = (\pi_2 \circ \tilde{\sigma})_* \mathcal{O}_{X_1}, \quad \textrm{that is}  \\
 \sigma_* (\mathcal{O}_A \oplus \mathcal{L}_1^{-1}) & = \pi_{2*} (\tilde{\sigma}_* \mathcal{O}_{X_1}),  \quad \; \; \textrm{that is}\\
 \mathcal{O}_A \oplus (\sigma_* \mathcal{L}_1^{-1}) & = \mathcal{O}_A \oplus  \mathcal{L}_2^{-1}.
\end{split}
\end{equation*}
The decomposition is preserved because the map $X_1 \to X_2$ is compatible with the involutions induced by the Albanese map, so we obtain $\sigma_* \mathcal{L}_1^{-1} \simeq \mathcal{L}_2^{-1}$ as desired.
If $\tilde{\sigma}$ is anti-holomorphic, it suffices to apply the same proof to the holomorphic diffeomorphism which is complex-conjugated to it.
\end{proof}
Summing up, Propositions \ref{prop:antiHolo and Holo} and \ref{prop:converse-holo} together imply

\begin{corollary} \label{cor:holo-antiholo}
With the same assumption as in Proposition \emph{\eqref{prop:converse-holo}}, there exists a holomorphic $($respectively, anti-holomorphic$)$ diffeomorphism $\tilde{\sigma} \colon X_1 \to X_2$ if and only if there exists a holomorphic $($respectively, anti-holomorphic$)$ element $\sigma \in \Kl(A, \, D)$ such that $\sigma^* \mathcal{L}_2 \simeq \mathcal{L}_1$.  
\end{corollary}

\subsection{The uniqueness of the abelian surface $A$}

We follow the notation of Section \ref{subsec:example-II}. If $\chi_i$ is any non-trivial element of the group $\varSigma=\{\chi_0, \, \chi_1,$ $\chi_2, \, \chi_3\}$  for the sake of brevity we will write $A_i$, $D_{A_i}$ and  $f_i \colon A_i \to A'$ instead of $A_{\chi_i}$, $D_{A_{\chi_i}}$ and $f_{\chi_i} \colon A_{\chi_i} \to A'$, respectively. We will also denote by $\mathcal{L}_i \in \textrm{Pic}^0(A')$ the $2$-torsion line bundle corresponding to $\chi_i$, so that  
$f_{i*} \mathcal{O}_{A_i} =  \mathcal{O}_{A'} \oplus \mathcal{L}_i^{-1}$.

\begin{proposition} \label{prop:isom-chi}
The abelian surfaces $A_1$, $A_2$, $A_3$ are pairwise isomorphic. More precisely, for all $i, \, j \in \{1, \, 2, \, 3\}$ there exists an isomorphism $\tilde{\gamma}_{ij} \colon A_j \to A_i$ such that $\tilde{\gamma}_{ij}^*D_{A_i}= D_{A_j}$.
\end{proposition}
\begin{proof}
By Proposition \ref{prop:antiHolo and Holo} it suffices to prove that there exists an automorphism $\gamma_{ij}  \in \aut(A', \, D_{A'})$ such that $\gamma_{ij}^* \, \mathcal{L}_i = \mathcal{L}_j$. 
Consider the linear automorphism $\gamma \colon V \to V$ whose action on the standard basis is
$\gamma(e_1)= - \zeta e_1, \quad \gamma(e_2)= e_1 + e_2.$ 
It preserves the lattice $\Lambda_{A'}$, in fact we have 
\begin{equation} \label{eq:action-varphi-lattice}
\gamma(\lambda_1)=e_1-\lambda_1, \quad \gamma
(\lambda_2)= \lambda_1+\lambda_2, \quad \gamma(e_1)=- \lambda_1, \quad \gamma(e_2)=e_1+e_2,
\end{equation}
so it descends to an automorphism of $A'$ that we still denote by $\gamma \colon A' \to A'$. An easy calculation shows that $\gamma$ is an element of order $3$ in $\mathrm{Aut}(A', \, D_{A'})$, so it induces by pull-back an action of $\langle \gamma \rangle \simeq \mathbb{Z}/ 3 \mathbb{Z}$ on $\mathrm{NS}(A')$. Such an action is obtained by composing a character $\chi \colon \Lambda_{A'} \to \{\pm1\}$ with \eqref{eq:action-varphi-lattice}, and it is straightforward to check that it restricts to an action on the subgroup $\varSigma$ (defined in (\ref{Characters})), namely the one generated by the cyclic permutation $(\chi_1 \; \chi_3 \; \chi_2).$
This shows that $\langle \gamma \rangle$ acts transitively on the non-trivial characters of $\varSigma$. Since the action of $\gamma$ on the characters $\chi_i$ corresponds to the pullback action on the corresponding $2$-torsion divisors $\mathcal{L}_i$,  by setting $\gamma_{13}=\gamma_{21}=\gamma_{32}=\gamma$ and $\gamma_{31}=\gamma_{12}=\gamma_{23}=\gamma^2$ we obtain $\gamma_{ij}^* \, \mathcal{L}_i = \mathcal{L}_j$, as desired.
\end{proof}

\subsection{A rigidity result for surfaces of type $II$}

Let us first recall the notions of deformation equivalence and global rigidity, see \cite[Section 1]{Ca13}.
\begin{definition} \label{def:def-and-rigidity} ${}$
\begin{itemize}
\item Two complex surfaces $S_1$, $S_2$ are said to be \emph{direct deformation equivalent} if there is a proper holomorphic submersion with connected fibres $f \colon \mathcal{Y} \to \mathbb{D}$, where $\mathcal{Y}$ is a complex manifold and $\mathbb{D} \subset \mathbb{C}$ is the unit disk, and moreover there are two fibres of $f$ biholomorphic to $S_1$ and $S_2$, respectively; 
\item two complex surfaces $S_1$, $S_2$ are said to be \emph{deformation equivalent} if they belong to the same deformation equivalence class, where by deformation equivalence we mean the equivalence relation generated by direct deformation equivalence;
\item a complex surface $S$ is called \emph{globally rigid} if its deformation equivalence class consists of $S$ only, i.e. if every surface which is deformation equivalent to $S$ is actually isomorphic to $S$.
\end{itemize}
\end{definition}

The following result is a characterization of the equianharmonic product, that can be found in  \cite[Proposition 5]{KH04}.
\begin{proposition} \label{prop:char-Hirzebruch}
Let $A'$ be an abelian surface containing four elliptic curves, 
that intersect pairwise at the origin $o'$ and not elsewhere. Then $A'$ is isomorphic to the equianharmonic product $E' \times E'$ and, up to the action of $\mathrm{Aut}(A')$, the four curves are $E_1'$, $E_2'$, $E_3'$, $E_4'$.    
\end{proposition}
A more conceptual proof of Proposition \ref{prop:char-Hirzebruch}, exploiting some results of Shioda and Mitani on abelian surfaces with maximal Picard number, can be found in \cite{Aide}.
Using Proposition \ref{prop:char-Hirzebruch} we obtain:
\begin{theorem} \label{thm:class-type II}
Let $S$ be a surface with $p_g(S)=q(S)=2$, $K_S^2=8$ and Albanese map $\alpha \colon S \to A$ of degree $2$. If $S$ belongs to type $II$, then the pair $(A, \, D_A)$ is isomorphic to an \'etale double cover of the pair $(A', \, D_{A'})$, where $A'$ is the equianharmonic product and $D_{A'}=E_1'+E_2'+E_3'+E_4'$. In particular, all surfaces of type $II$ arise as in Proposition \emph{\ref{prop:type-II}}. Finally, all surfaces of type $II$ are globally rigid.
\end{theorem}
\begin{proof}
Let us consider the Stein factorization $\alpha_X \colon X \to A$ of the Albanese map $\alpha \colon S \to A$; then $\alpha_X$ is a finite double cover branched over $D_A$.

By Proposition \ref{prop:branch-type-II} we have $D_A=E_1+E_2+E_3+E_4$, where the $E_i$ are four elliptic curves intersecting pairwise transversally at two points $p_1$, $p_2$ and not elsewhere. Up to a translation, we may assume that $p_1$ coincides with the origin of $o \in A$. Then $p_2= a$, where $a$ is a non-zero, $2$-torsion point of $A$ (in fact the $E_i$ are subgroups of $A$, so the same is true for their intersection $\{o, \, a\}$).  

If we consider the abelian surface $A':= A/ \langle a \rangle$, then the projection $f \colon A \to A'$ is an isogeny of degree $2$. Moreover, setting $E_i' :=f(E_i)$, we see that $E_i', \ldots, E_4'$ are four elliptic curves intersecting pairwise transversally at the origin $o' \in A'$ and not elsewhere. Then the claim about $A'$ and $D_{A'}$ follows from Proposition \ref{prop:char-Hirzebruch}.

Let us finally provide our rigidity argument. First of all, we observe that surfaces of type $II$ cannot be specialization of surfaces of type $I$, for instance because every flat deformation of the Albanese map must preserve the arithmetic genus of the branch divisor, and so must preserve $D_A^2$ (in fact, we can state the stronger result that every flat limit of surfaces of type $I$ is still a surface of type $I$, because being isogenous to a higher product is a topological condition by \cite[Theorem 3.4]{Cat00}). From this, since there are finitely many possibilities for both the double covers $f \colon A \to A'$ and $a_X \colon X \to A$, it follows that $S$, being the minimal desingularization of $X$, belongs to only finitely many isomorphism classes. Therefore $S$ is globally rigid, because by \cite{G77} the moduli space of surfaces of general type is separated.
\end{proof}

\subsection{The groups $\protect\aut(A,\, D_A)$ and $\protect\Kl(A, \, D_A)$}

In the sequel we will write 
$A:=V/\Lambda_A $
in order to denote any of the pairwise isomorphic abelian surfaces $A_1$, $A_2$, $A_3$, see Proposition \ref{prop:isom-chi}. We choose for instance $\Lambda_A=\ker \chi_1$. By \eqref{eq:bases-ker-chi} we have 
\begin{equation}
\Lambda_A = \mathbb{Z} \mathbf{e}_1 \oplus \mathbb{Z} \mathbf{e}_2 \oplus  \mathbb{Z}  \mathbf{e}_3 \oplus \mathbb{Z}  \mathbf{e}_4,   
\end{equation}   
where
\begin{equation} \label{eq:e1-e2}
\mathbf{e}_1:=e_1, \quad  \mathbf{e}_2:=\lambda_1+e_2, \quad \mathbf{e}_3:=\lambda_2+e_2, \quad  \mathbf{e}_4:=2e_2.
\end{equation}
Note that $\mathbf{e}_1=(1, \, 0)$ and $\mathbf{e}_2 = (\zeta, \, 1)$ form a basis for $V$. 
\begin{remark} \label{rem:a}
It is straightforward to check that the class of the point $\zeta e_1 =  (\zeta, \, 0)$ in $A_1$ is contained in all the curves $E_1, \ldots, E_4$, so we obtain $a=\zeta e_1 + \Lambda_A$, where $a$ is the $2$-torsion point defined in the proof of Theorem \ref{thm:class-type II}.
\end{remark}


Let $\Gamma_{\zeta}$ and $E'$ be as in \eqref{eq:curve-E'} and set
\begin{equation*}
E'':=\mathbb{C}/\Gamma_{2 \zeta}, \quad \Gamma_{2 \zeta} := \mathbb{Z}[2 \zeta].
\end{equation*}
The next result implies that $A$ is actually isomorphic to the product  $E'' \times E'$.

\begin{lemma} \label{lem:Lambda-A}
We have $\Lambda_A = \Gamma_{2 \zeta} \, \mathbf{e}_1 \oplus \Gamma_{\zeta} \, \mathbf{e}_2.$
\end{lemma}
\begin{proof}
We check that the base-change matrix between the $\mathbb{Q}$-bases $\mathbf{e}_1, \mathbf{e}_2, \mathbf{e}_3, \mathbf{e}_4$ and $\mathbf{e}_1, \mathbf{e}_2, 2\zeta \mathbf{e}_1, \zeta \mathbf{e}_2$ of $H_1 (A,\mathbb{Q})$ is in $\mathsf{GL}(4,\mathbb{Z})$.
\end{proof}
We will use  Lemma \ref{lem:Lambda-A} in order to describe the groups $\aut(A, \, D_A)$ and $\Kl (A, \, D_A)$. 
In what follows,
we will identify an automorphism $A \to A$ with the matrix of its analytic
representation $V \to V$ with respect to the standard basis $\{e_{1},\, e_{2} \}$. Moreover, we will write $\tau=\tau_a \colon A \to A$ for the translation by the $2$-torsion
point $a= \zeta e_1 + \Lambda_A$. 

\begin{proposition} \label{prop:The-group-Aut(A,D)}
The following holds.
\begin{itemize}
\item[$\boldsymbol{(a)}$]
We have
\begin{equation} \label{eq:aut(A,D)}
\mathrm{Aut}(A, \, D_A) = \mathrm{Aut}_0(A, \, D_A) \times \mathbb{Z}/2 \mathbb{Z},
\end{equation}
where $\mathbb{Z}/2 \mathbb{Z}$ is generated by the translation $\tau$, whereas $\mathrm{Aut}_0(A, \, D_A)$ is the subgroup of group automorphisms of $A$ generated by the elements
\begin{equation} \label{eq:g2-g3}
g_{2}=\left(\begin{array}{cc}
\zeta & -1\\
\zeta & - \zeta
\end{array}\right),\quad g_{3}=\left(\begin{array}{cc}
0 & \zeta-1\\
1-\zeta & \zeta-1
\end{array}\right).
\end{equation}
As an abstract group, $\mathrm{Aut}_0(A, \, D_A)$ is isomorphic to $\mathsf{SL}(2, \, \mathbb{F}_3);$ in particular, its order is $24$.
\item[$\boldsymbol{(b)}$]
The group $\Kl(A, \, D_A)$ is generated by $\aut(A, \, D_A)$ together with the anti-holomorphic involution $\sigma \colon A \to A$ induced by the $\mathbb{C}$-antilinear involution of $V$ given by  
\begin{equation} \label{eq:sigma}
(z_{1}, \, z_{2}) \mapsto ((\zeta-1)\bar{z}_{2}, \, (\zeta-1)\bar{z}_{1}).
\end{equation}
Furthermore, the two involutions $\tau$ and $\sigma$ commute, so that we can write 
\begin{equation} \label{eq:Kl(A,D)}
\Kl(A, \, D_A) = \mathrm{Kl}_0(A, \, D_A) \times\ZZ/2\ZZ,
\end{equation}
where $\mathrm{Kl}_0(A, \, D_A) $ contains $\mathrm{Aut}_0(A, \, D_A)$ as a subgroup of index $2$.
\end{itemize}
\end{proposition}
\begin{proof}
$\boldsymbol{(a)}$ Let us work using the basis $\{\mathbf{e}_1, \, \mathbf{e_2}\}$ of $V$ defined in $\eqref{eq:e1-e2}$. With respect to this basis, using \eqref{eq:four-complex-lines} we see that the four elliptic curves $E_{1},\dots,E_{4}$ have tangent spaces  
\begin{equation*}
\begin{aligned}
V_1 & = \mathrm{span}(\mathbf{e}_1), & V_2  = \mathrm{span}(-\zeta \mathbf{e}_1 + \mathbf{e}_2), \\
V_3 & = \mathrm{span}((1-\zeta) \mathbf{e}_1 + \mathbf{e}_2), & V_4  = \mathrm{span}((1- 2\zeta) \mathbf{e}_1 + \mathbf{e}_2). \\
\end{aligned}
\end{equation*}
Then, up to the translation $\tau$, we are looking at the subgroup $\mathrm{Aut}_0(A, \, D_A)$ of the group automorphisms of $A$ whose elements have matrix representation preserving the set of four points $ \mathscr{P}=\{P_1, \, P_2, \, P_3, \, P_4 \} \subset \mathbb{P}^1$, where     
\begin{equation*}
P_1=[1:\, 0], \quad P_2=[-\zeta: \, 1], \quad P_3=[1-\zeta: \, 1],\quad P_4=[1-2\zeta: \, 1].
\end{equation*}
The cross ratio $(P_1, \, P_2 ,\, P_3, \, P_4)$ equals $\zeta^{-1}$, hence $\mathscr{P}$ is an equianharmonic quadruple and so the group $\mathrm{PGL}(2, \, \mathbb{C})$ acts on it as the alternating group $\mathsf{A}_4$, see 
\cite[p. 817]{Maier07}. Such a group can be presented as
\begin{equation} \label{eq:presentation-A4}
\mathsf{A}_4 = \langle \alpha, \, \beta \; | \; \alpha^2 = \beta^3 = (\alpha \beta)^3 =1 \rangle,
\end{equation}
where $\alpha = (13)(24)$ and $\beta=(123)$, so we need to find matrices $\tilde{g}_2, \, \tilde{g}_3 \in \mathrm{GL}(2, \, \mathbb{C})$, acting as an isomorphism on the lattice $\Lambda_A$ and inducing the permutations $(P_1 \, P_3)(P_2 \, P_4)$ and $(P_1 \, P_2 \, P_3)$ on $\mathscr{P}$, respectively. 
Using Lemma \ref{lem:Lambda-A}, we see that $\tilde{g} \in \mathrm{GL}(2, \, \mathbb{C})$ preserves $\Lambda_A$ if and only if it has the form
\begin{equation*}
\tilde{g}=\left(\begin{array}{cc}
a_{11} & a_{12}\\
a_{21} & a_{22}
\end{array}\right), \quad \mathrm{with}\; a_{11} \in \Gamma_{2 \zeta}, \; \; a_{12} \in 2 \Gamma_{\zeta}, \; \; a_{21}, \, a_{22} \in \Gamma_{\zeta},
\end{equation*}
and its determinant belongs to the  group of units of $\Gamma_{\zeta}$, namely $\{\pm 1, \, \pm \zeta, \, \pm \zeta^2 \}$.

Now an elementary computation yields the matrices  $\pm \tilde{g}_{2}, \, \pm \tilde{g}_{3}$, where 
\begin{equation*}
\tilde{g}_{2}=\left(\begin{array}{cc}
1 & 2 \zeta -2\\
\zeta & -1
\end{array}\right), \quad \tilde{g}_{3}=\left(\begin{array}{cc}
-1 & 0\\
1- \zeta & \zeta
\end{array}\right),  
\end{equation*}
and from this we can obtain the matrix representations $g_2$, $g_3$ of our automorphisms in the basis $\{e_1, \, e_2 \}$ of $V$ by taking 
\begin{equation*}
g_i = N \tilde{g}_i N^{-1}, \quad \mathrm{with} \;  N =\left(\begin{array}{cc}
1 & \zeta\\
0 & 1
\end{array}\right).
\end{equation*}
This gives  \eqref{eq:g2-g3}. Setting $h=-I_2$ and lifting the presentation \eqref{eq:presentation-A4} we get the presentation
\begin{equation} 
\mathrm{Aut}_0(A, \, D_A)=\langle g_2, \, g_3, \, h \; | \;  h^2=1, \, \, g_2^2 = g_3^3 = (g_2g_3)^3= h \rangle,
\end{equation}
showing the isomorphism $\mathrm{Aut}_0(A, \, D_A) \simeq \mathsf{SL}(2, \, \mathbb{F}_3)$. 

Finally, a standard computation shows that $\tau$ commutes with both $g_2$ and $g_3$. Since $\mathrm{Aut}(A)$ is the semidirect product of the translation group of $A$ by the group automorphisms, it follows that $\mathrm{Aut}(A, \, D_A)$ is the direct product of $\langle \tau \rangle \simeq \mathbb{Z}/2 \mathbb{Z}$ by $\mathrm{Aut}_0(A, \, D_A)$, hence we obtain \eqref{eq:aut(A,D)}.
\\ \\
$\boldsymbol{(b)}$ Let us consider the $\mathbb{C}$-antilinear map  $V \to V$
\begin{equation*}
(z_{1}, \, z_{2})\mapsto(\bar{z}_{1}, \, (\zeta-1)\bar{z}_{1} + \zeta \bar{z}_{2}), 
\end{equation*}
expressed with respect to the basis $\{\mathbf{e}_1, \, \mathbf{e}_2 \}$.  
It preserves the lattice $\Lambda_A$ and so it defines an anti-holomorphic involution $\sigma\colon A \to A$, 
inducing the transposition $(P_1 \, P_2)$ on the set $\mathscr{P}=\{P_1, \, P_2, \, P_3, \, P_4\}$. Since $\aut(A,D)$
has index at most $2$ in $\Kl(A, \, D_A)$, it follows that $\aut(A, \, D_A)$ and
$\sigma$ generate $\Kl(A, \, D_A)$. Moreover, a change of coordinates allows us to come back to the basis $\{e_1, \, e_2\}$ and to obtain the expression of $\sigma$ given in \eqref{eq:sigma}.
The subgroup $\mathrm{Kl}_0(A, \, D_A)$ generated by $g_2, \, g_3$ and the involution $\sigma$ contains $\mathrm{Aut}_0(A, \, D_A)$ as a subgroup of index $2$; a straightforward computation now shows that $[\tau, \, \sigma]=0$, so \eqref{eq:Kl(A,D)} follows from \eqref{eq:aut(A,D)} and the proof is complete.
\end{proof}

\begin{remark} \label{rmk:central-extensions}
The proof of Proposition \ref{prop:The-group-Aut(A,D)} also shows that there are central extensions
\begin{equation*} \label{eq:central-extensions}
\begin{split}
&1 \to \langle -I_2 \rangle \longrightarrow \mathrm{Aut}_0(A, \, D_A) \longrightarrow \mathsf{A}_4 \longrightarrow 0, \\
&1 \to \langle -I_2 \rangle \longrightarrow \mathrm{Kl}_0(A, \, D_A) \longrightarrow \mathsf{S}_4 \longrightarrow 0
\end{split}
\end{equation*}
such that $-I_2=[g_2, \, g_3]^2$, where the commutator is defined as $[x, \, y] := xyx^{-1}y^{-1}$. In fact, $\mathrm{KL}_0(A, \, D_A)$ is isomorphic as an abstract group to $\mathsf{GL}(2, \, \mathbb{F}_3)$, see the proof of Theorem \ref{thm:The-number-of-Surfaces}. 
\end{remark}

\subsection{The action of $\protect\Kl(A,\, D_A)$ on the square roots of $\mathcal{O}_A(D_A)$}

The main result of this subsection is the following
\begin{theorem}
\label{thm:The-number-of-Surfaces} Up to isomorphism, there exist exactly two surfaces of type II. These surfaces $S_{1}, \, S_{2}$ have conjugated complex structures, in other words there exists an anti-holomorphic diffeomorphism $S_1 \to S_2$.
\end{theorem}
In order to prove this result, we must study the action of the groups $\mathrm{Aut}(A, \, D_A)$ and $\mathrm{Kl}(A, \, D_A)$ on the sixteen square roots $\mathcal{L}_1, \ldots, \mathcal{L}_{16}$ of the line bundle $\mathcal{O}_A(D_A) \in \mathrm{Pic}(A)$. The Appell-Humbert data of such square roots are described in the following
\begin{proposition} \label{prop:square-roots-DA}
For $k \in \{1, \ldots, 16\},$ we have
$\mathcal{L}_k = \mathcal{L} \left(\frac{1}{2}\,h_A, \, \psi_k \right)$
where
\begin{itemize}
\item $h_A \colon V \times V \to \mathbb{C}$ is the hermitian form on $V$ whose associated alternating form $\mathrm{Im}\, h_A$ assumes the following values at the generators $\mathbf{e}_1, \ldots, \mathbf{e}_4$ of $\Lambda_A:$
\begin{table}[H] 
\begin{center}
\begin{tabular}{c|c|c|c|c|c|c} 
$(\cdot, \, \cdot)$   & $(\mathbf{e}_1, \, \mathbf{e}_2)$ & $(\mathbf{e}_1, \, \mathbf{e}_3)$ & $(\mathbf{e}_1, \, \mathbf{e}_4)$ & $(\mathbf{e}_2, \, \mathbf{e}_3)$ & $(\mathbf{e}_2, \, \mathbf{e}_4) $ & $(\mathbf{e}_3, \, \mathbf{e}_4)$ \\
\hline
$\mathrm{Im} \; h_A(\cdot, \, \cdot)$ & $-4$ & $0$ & $-2$ & $-6$ & $-4$ & $6$ \\
\end{tabular} \caption{The values of $\mathrm{Im}\, h_A$ at the generators of $\Lambda_{A}$} \label{table:imh_A}  
\end{center}
\end{table}
\item Using the notation $\psi_k := (\psi_k(\mathbf{e}_1), \, \psi_k(\mathbf{e}_2), \, \psi_k(\mathbf{e}_3), \, \psi_k(\mathbf{e}_4)),$
the semicharacters $\psi_k \colon \Lambda_A \to \mathbb{C}^*$ are as follows$:$ 
\begin{equation} \label{eq:semi-characters}
\begin{array}{lll}
\psi_1 \ := (i, \, 1, \, i, \, 1), & & \\
\psi_2\ :=(-i, \, -1, \, i, -1),    & \psi_3\ := (i, \, -1, \, -i, \, 1),  & \psi_4\ :=(-i, \, 1, \, -i, \, -1),\\ 
\psi_5\ :=(i, \, 1, \, -i, \, 1),   &  \psi_6\ :=(-i, \, 1, \, i, \,1),    & \psi_7\ :=(-i, \, 1,\, -i, \, 1),\\ 
\psi_8\ :=(i, \, 1, \, i, \, -1),   &  \psi_9\ :=(i, \, -1, \, i, \, -1),  & \psi_{10}\ :=(-i, \, 1, \, i, \,-1),\\
\psi_{11} :=(i, \, 1,\, -i, \, -1), & \psi_{12}:=(-i, \, -1, \, -i, \, 1), & \psi_{13} :=(i, \, -1, \, i, \, 1),\\
\psi_{14}:=(i, \, -1, \, -i, \,-1), &  \psi_{15}:=(-i, \, -1,\, i, \, 1),  &  \psi_{16} :=(-i, \, -1, \, -i, \, -1).
\end{array}
\end{equation}  
\end{itemize}
\end{proposition}
\begin{proof}
Let us consider the double cover $f \colon A \to A'$.
If the hermitian form $h \colon V \times V \to \mathbb{C}$ and the semicharacter $\chi_{D_{A'}} \colon \Lambda_{A'} \to \{ \pm1 \}$ are as in Proposition \ref{pro:class_DA'} and Table \ref{table:imh}, then we have $\mathcal{O}_A(D_A) = \mathcal{L}(h_A, \, \chi_{D_A})$, where 
$h_A=f^*h, \quad \chi_{D_A} = f^* \chi_{D_{A'}}.$
From this, using \eqref{eq:e1-e2} we can compute the values of the alternating form $\mathrm{Im}\, h_A$ and of the semicharacter  $\chi_{D_A}$ at $\mathbf{e_1}, \ldots, \, \mathbf{e}_4$, obtaining Table \ref{table:imh_A} and 
\begin{equation*}
\chi_{D_A} = (\chi_{D_A}(\mathbf{e}_1), \, \chi_{D_A}(\mathbf{e}_2), \, \chi_{D_A}(\mathbf{e}_3), \, \chi_{D_A}(\mathbf{e}_4)) = (-1, \, 1, \, -1, \, 1).
\end{equation*} 
Then, setting $\mathcal{L}_k = \mathcal{L}(h_k, \, \psi_k)$, the equality $\mathcal{L}_k^{\otimes 2} = \mathcal{O}_A(D_A)$ implies
\begin{equation*}
2h_k = h_A, \quad \psi_k^2 = \chi_{D_A},
\end{equation*}  
hence $h_k = \frac{1}{2} h_A$ for all $k$. Moreover we can set $\psi_1 = (i, \, 1, \, i, \, 1)$, whereas the remaining $15$ semicharacters $\psi_k$ are obtained by multiplying $\psi_1$ by the $15$ non-trivial characters $\Lambda_A \to \{ \pm 1 \}$.
\end{proof} 
The hermitian form $h_A$ is $\mathrm{Kl}(A, \, D_A)$-invariant (accordingly with the fact that the divisor $D_A$ is so), hence the action of $\mathrm{Kl}(A, \, D_A)$ on the set $\{\mathcal{L}_1, \ldots, \mathcal{L}_{16} \}$ is completely determined by its permutation action on the set $\{\psi_1, \ldots, \psi_{16} \}$, namely  
\begin{equation*}
 \varrho \colon \mathrm{Kl}(A, D_A) \longrightarrow \mathsf{Perm}(\psi_1, \ldots, \psi_{16}), \quad 
 \varrho(g) (\psi_k) := g^* \psi_k.
\end{equation*}
Let us now identify the group $ \mathsf{Perm}(\psi_1, \ldots, \psi_{16})$ with the symmetric group $\mathsf{S}_{16}$ on the symbols $\{1, \ldots, 16\}$, where we multiply permutations from the left to the right, for instance $(1 \, 2)(1\, 3)=(1\,2\,3)$. Then we get the following 
\begin{proposition} \label{prop:permutation-action}
With the notation of \emph{Proposition \ref{prop:The-group-Aut(A,D)}}, we have
\begin{equation*}
\begin{split}
\varrho(g_2)& =(1 \; \, 13 \; \,  7\;\, 12)(2 \;\, 9 \;\, 14 \;\, 16)(3 \;\, 5 \;\, 15\;\, 6)(4 \;\, 11 \;\, 8 \;\, 10), \\
\varrho(g_3)& = (1 \; \, 13\; \, 5 \; \, 7 \; \, 12 \; \, 6)(2 \; \, 4 \; \, 11 \; \, 14 \; \, 8 \; \, 10) (3 \; \, 15) (9 \; \, 16), \\
\varrho(-I_2)=\varrho(g_2^2) = \varrho(g_3^3)= \varrho(\tau)& = (1 \; \, 7)(2 \; \, 14) (3 \; \, 15)(4 \; \, 8)(5 \; \, 6)(9 \; \, 16)(10 \; \, 11)(12 \; \, 13), \\
\varrho(\sigma) & = (1 \; \, 14)(2 \; \, 7) (3 \; \, 16)(4 \; \, 5)(6 \; \, 8)(9 \; \, 15)(10 \; \, 12)(11 \; \, 13). 
\end{split}
\end{equation*}
\end{proposition}
\begin{proof}
Using the explicit expressions given in Proposition \ref{prop:The-group-Aut(A,D)}, by a standard computation we can check that  $g_2$, $g_3$, $\tau$ and $\sigma$ send the ordered basis $\{\mathbf{e}_1, \, \mathbf{e}_2, \, \mathbf{e}_3, \, \mathbf{e}_4 \}$ of $\Lambda_A$ to the bases
\begin{equation} \label{eq:mathbf-e}
\begin{split}
& \{\mathbf{e}_2 + \mathbf{e}_3 - \mathbf{e}_4, \; -2 \mathbf{e}_1+ \mathbf{e}_2  - \mathbf{e}_4, \; -\mathbf{e}_1 - \mathbf{e}_2 - 2 \mathbf{e}_3+ 2\mathbf{e}_4, \; -2\mathbf{e}_1 - 2 \mathbf{e}_3+ \mathbf{e}_4 \}, \\ 
 & \{-\mathbf{e}_3 + \mathbf{e}_4, \; -\mathbf{e}_1 +\mathbf{e}_2 + \mathbf{e}_3- \mathbf{e}_4, \; - 2\mathbf{e}_1 + \mathbf{e}_2 + \mathbf{e}_3 -2 \mathbf{e}_4, \; - 2\mathbf{e}_1 + 2 \mathbf{e}_2 + 2 \mathbf{e}_3 -3 \mathbf{e}_4 \}, \\
& \{-\mathbf{e}_3 \; - \mathbf{e}_2, \; - \mathbf{e}_3, \; - \mathbf{e}_4 \}, \\
& \{\mathbf{e}_3 - \mathbf{e}_4, \; -\mathbf{e}_1 +\mathbf{e}_2 +\mathbf{e}_3- \mathbf{e}_4, \; -\mathbf{e}_1 +2 \mathbf{e}_2 - \mathbf{e}_4, \; -2 \mathbf{e}_1 +2 \mathbf{e}_2 - \mathbf{e}_4 \},
\end{split}
\end{equation}
respectively. For any $g \in \mathrm{Aut}(A, \, D_A)$, calling $G_{\Lambda} \colon \Lambda_A \to \Lambda_A$ the corresponding rational representation we have $\varrho(g)(\psi_k) = \psi_k \circ G_{\Lambda}$, whereas $\varrho(\sigma)(\psi_k) = \overline{\psi_k \circ \mathfrak{S}_{\Lambda}}$ (see \eqref{eq:L-anti-holo-pullback}). Then another long but straightforward calculation using \eqref{eq:semi-characters} and \eqref{eq:mathbf-e}  concludes the proof.
\end{proof}
 
We are now ready to give the \\

\emph{Proof of Theorem $\mathrm{\ref{thm:The-number-of-Surfaces}}$}.  Since any surface $S$ of type $II$ is the minimal desingularization of a double cover $f \colon S \to A$, branched over $D_A$, by  Proposition \ref{prop:converse-holo} it follows that the number of surfaces of type $II$ up to isomorphisms (respectively, up to holomorphic and anti-holomorphic diffeomorphisms) equals the number of orbits for the permutation action of $\mathrm{Aut}(A, \, D_A)$ (respectively, of $\mathrm{Kl}(A, \, D_A)$) on the set $\{\mathcal{L}_1, \ldots, \mathcal{L}_{16}\}$ of the sixteen square roots of $\mathcal{O}_A(D_A)$. We have seen that such an action is  determined by the permutation action on the set of sixteen semicharacters $\{\psi_1, \ldots, \psi_{16}\}$, so we only have to compute the number of orbits for the subgroup of $\mathsf{S}_{16}$ whose generators are described in Proposition \ref{prop:permutation-action}.     
This can be done by hand, but it is easier to write a short script using the Computer Algebra System  \verb|GAP4| (\cite{GAP4}):
\begin{Verbatim}[commandchars=\\\{\}]
g2:=(1, 13, 7, 12)(2, 9, 14, 16)(3, 5, 15, 6)(4, 11, 8, 10);; 
g3:=(1, 13, 5, 7, 12, 6)(2, 4, 11, 14, 8, 10)(3, 15)(9, 16);;
sigma:=(1, 14)(2, 7)(3, 16)(4, 5)(6, 8)(9, 15)(10, 12)(11, 13);; 
Aut:=Group(g2, g3);; 
Kl:=Group(g2, g3, sigma);; 
StructureDescription(Aut);
\textcolor{red}{"SL(2,3)"}
OrbitsPerms(Aut, [ 1 .. 16 ]);
\textcolor{red}{[ [ 1, 7, 12, 13, 3, 15, 5, 6 ], [ 2, 14, 16, 9, 10, 11, 4, 8 ] ]} 
StructureDescription(Kl);
\textcolor{red}{"GL(2,3)"}
OrbitsPerms(Kl, [1..16]);
\textcolor{red}{[ [ 1, 7, 12, 13, 15, 3, 6, 5, 14, 2, 10, 11, 9, 16, 8, 4 ] ]}
\end{Verbatim} 
The first two output lines (in red) show that $\varrho$ induces an embedding of $\mathrm{Aut}_0(A, \, D_A)$ in $\mathsf{S}_{16}$ (note that $\rho$ is a group homomorphism because of our convention on the multiplication on $\mathsf{S}_{16}$), and that the corresponding permutation subgroup has precisely two orbits. Therefore there are exactly two surfaces $S_1$, $S_2$ of type $II$, up to isomorphisms. Analogously, the last two output lines show that $\varrho$ induces an embedding of $\mathrm{Kl}_0(A, \, D_A)$ in $\mathsf{S}_{16}$,
and the corresponding permutation subgroup has only one orbit. This means that there exists an anti-holomorphic diffeomorphism $S_1 \to S_2$, hence these surfaces are not isomorphic, but they have conjugated complex structures. \hfill $\Box$

\bigskip

Let us finally show that surfaces of type $II$ are not uniformized by the bidisk (unlike surfaces of type $I$, see Corollary \ref{cor:univ-I}).
\begin{proposition} \label{prop:univ-II}
Let $S$ be a surface of type $II$ and $\widetilde{S} \to S$ its universal cover. Then $\widetilde{S}$ is \emph{not} biholomorphic to $\mathbb{H} \times \mathbb{H}$.  
\end{proposition}
\begin{proof} 
Looking at diagram  \eqref{dia.resolution} in Section 1, we see that, in case $II$, the map $\varphi \colon B \to A$ is the blow-up of $A$ at the two quadruple points $p_1, \, p_2$ of the curve $D_A$ and that  $\bar{S}=S$. Moreover, considering $\beta \colon S \to B$ we have  
\begin{equation*}
\beta ^* D_B = C_1 + C_2 + C_3 + C_4, 
\end{equation*}
where the $C_i$ are (pairwise disjoint) elliptic curves with $C_i^2=-1$.
The embedding $C_{i} \to  S$, composed with the universal cover $\mathbb{C} \to C_i$, gives a 
non-constant holomorphic map $\CC \to S$, that in turn lifts to a non-constant holomorphic map $\mathbb{C} \to \widetilde{S}$. If $\widetilde{S}$ were isomorphic to $\mathbb{H} \times \mathbb{H}$, projecting onto one of the two factors we would obtain a non-constant holomorphic map $\mathbb{C} \to \mathbb{H}$, whose existence would contradict Liouville's theorem because $\mathbb{H}$ is biholomorphic to the bounded domain $\mathbb{D} = \{z \in \mathbb{C} \; : \, |z| < 1 \}$.
\end{proof}

\subsection{Concluding remarks} \label{subsec:final}

\begin{remark}
In the argument in the proof of Proposition \ref{prop:univ-II}, we could have used one of the elliptic curves $Z_1, \, Z_2$ instead of the $C_i$ (see Remark \ref{rem:sing-Stein}).
\end{remark}

\begin{remark}\label{openball}
Denoting by $\chi_{\textrm{top}}$ the topological Euler number, we have 
\begin{equation*}
\bigg(K_{S}+\sum_{i=1}^{4}C_{i} \bigg)^{2}=12=3 \, \chi_{\textrm{top}} \bigg(S - \sum_{i=1}^{4}C_{i} \bigg)
\end{equation*}
and
\begin{equation*}
\bigg(K_{S}+\sum_{i=1}^{2}Z_{i} \bigg)^{2}=12=3 \, \chi_{\textrm{top}} \bigg(S - \sum_{i=1}^{2}Z_{i} \bigg).
\end{equation*}
This implies that the open surfaces $S - \sum_{i=1}^{4}C_{i}$ and  $S - \sum_{i=1}^{2}Z_{i}$ both have the structure of a complex ball-quotient, see \cite{Ro14} for references and further details. 
\end{remark}

\begin{remark} \label{rem:def-diff}
The two non-isomorphic surfaces of type $II$ exhibit a new occurrence of the so-called \verb|Diff|$\nRightarrow$\verb|Def| phenomenon, meaning that their diffeomorphism type does not determine their deformation class. In fact, they are (anti-holomorphically) diffeomorphic,  but not deformation equivalent since they are rigid. See \cite{Man01}, \cite{ KK}, \cite{Ca03}, \cite{ CaWa07} for further examples of this situation. 
\end{remark}

\begin{remark} \label{rem:Shimura}
We can also give the following different proof of  Proposition \ref{prop:univ-II}, based on the same argument used in \cite[proof of Proposition 10]{RRS18}. We are indebted  to F. Catanese for several comments and suggestions on this topic. If a surface $S$ is uniformized by $\mathbb{H} \times \mathbb{H}$, then it is the quotient of $\mathbb{H} \times \mathbb{H}$ by a cocompact subgroup $\Gamma$ of $\mathrm{Aut}(\mathbb{H} \times \mathbb{H})=\mathrm{Aut}(\mathbb{H}) \rtimes (\mathbb{Z}/2 \mathbb{Z})$ acting freely. If $\Gamma$ is reducible in the sense of \cite[Theorem 1]{Shi63} it follows that $S$ is isogenous to a product; in particular, if $p_g=q=2$, $K_S^2=8$ then $S$ is of type $I$ by the classification given in \cite{Pe11}. Thus, if $S$ is of type $II$ then $\Gamma$ must be irreducible and so, by using \cite[Theorem 7.2]{MS63}, we infer
\begin{equation*}
q(S)=\frac{1}{2}b_1(S)= \frac{1}{2}b_1(\mathbb{P}^1 \times \mathbb{P}^1)=0,
\end{equation*}
a contradiction. 
\end{remark}

\begin{remark} \label{remGeomConstr}
It is possible to give a different geometric construction of the abelian surfaces $A'$, $A$ and of the divisor $D_A$ as follows.  Unfortunately, at present we do not know how to recover the $2$-divisibility of the curve $D_A$ in $\mathrm{Pic}(A)$ by using this alternative approach.

Let $F_1, \, F_2, \, F_3$ and $G_1, \,G_2, \, G_3$ be general fibres of the two rulings $f, \, g \colon\PP^1\times\PP^1 \to \mathbb{P}^1$, respectively; then 
the two reducible divisors $F_1+F_2+F_3$ and $G_1+G_2+G_3$ meet at nine distinct points. Consider three of these points, say  $p_1, \, p_2, \, p_3$, with the property that  
each $F_i$ and each $G_i$ contain exactly one of them. Then there exists precisely one (smooth) curve $C_1$ of bidegree $(1, \, 1)$ passing through $p_1, \, p_2, \, p_3$.
Similarly, if we choose three other points $q_1, \, q_2, \, q_3\notin\{p_1, \, p_2, \, p_3\}$ with the same property, there exist a unique curve $C_2$ of bidegree $(1, \, 1)$ passing through $q_1, \, q_2, \, q_3.$
The curves $C_1$ and $C_2$ meet at two points, say $r_1, \, r_2$, different from the points $p_i$ and $q_i$. 

Let us call $F_4$ and $G_4$ the fibres of $f$ and $g$ passing through one of these two points, say $r_1$. Then the reducible curve $B$ of bidegree $(4, \, 4)$ defined as 
\begin{equation*}
B=F_1+ \cdots +F_4+G_1+ \cdots +G_4
\end{equation*}
has sixteen ordinary double points as only singularities, and
the double cover $\phi \colon Q'\longrightarrow \PP^1\times\PP^1$ branched over 
$B$ gives a singular Kummer surface $Q'$; let us write $A'$ for the associate abelian surface. 
We can easily show that
\begin{equation*}
\phi^*C_1 = C_{11}+C_{12}, \quad \phi^*C_2 = C_{21}+C_{22},
\end{equation*}
where all the $C_{ij}$ are smooth and irreducible. Moreover we see that $C_{11}$ and $C_{22}$ intersect at exactly one point, which is a node of $Q'.$
Writing 
\begin{equation*}
\phi^*F_4 = 2 \widehat F_4, \, \quad \phi^*G_4 = 2 \widehat G_4,
\end{equation*}
we see that the rational curves $C_{11},$ $C_{22},$ $\widehat F_4, \, \widehat G_4$ meet at one node of $Q'$ and that each of them contains precisely four nodes of $Q'.$ Hence the pullback of these curves via the double cover $A' \to Q'$ yields four elliptic curves in $A'$ intersecting pairwise and transversally at a single point. 

Let us choose now $i, \, j, \, h, \, k \in \{1, \, 2, \, 3, \, 4 \}$, with $i\not=j$ and $h\not=k$, and consider the eight nodes of $B$ lying in the smooth locus of the curve $H=F_i+F_j+G_h+G_k$.
The $2$-divisibility of $H$ in $\mathrm{Pic}(\mathbb{P}^1 \times \mathbb{P}^1)$ implies that the corresponding set $\Xi$ of eight nodes in the Kummer surface $Q'$ is $2$-divisible, so we can consider the double cover $Q\longrightarrow Q'$ branched over $\Xi$.
The surface $Q$ is again a singular Kummer surface and, calling $A$ the abelian surface associate with $Q$,  we obtain a degree $2$ isogeny $A\longrightarrow A'$.
We can choose (in three different ways) $i, \, j, \, h, \, k$ so that each of the four curves $C_{11},$ $C_{22},$ $\widehat F_4$ and $\widehat G_4$ contains exactly two points of $\Xi$.
Therefore the pullback of these curves to $Q$ are rational curves, all passing through two of the nodes of $Q$ and containing four nodes each. This in turn gives four elliptic curves in $A$ meeting at two common points and not elsewhere, and the union of these curves is the desired divisor $D_A$.
\end{remark}


\bibliography{References}

\bigskip

\noindent Francesco Polizzi
\vspace{0.1cm}
\\ Dipartimento di Matematica e Informatica, Universit\`{a} della Calabria
\\ Cubo 30B, 87036 Arcavacata di Rende (Cosenza), Italy
\\ \verb|polizzi@mat.unical.it| 

\bigskip

\noindent Carlos Rito
\vspace{0.1cm}
\\{\it Permanent address:}
\\ Universidade de Tr\'as-os-Montes e Alto Douro, UTAD
\\ Quinta de Prados
\\ 5000-801 Vila Real, Portugal
\\ www.utad.pt, \verb|crito@utad.pt|
\\{\it Temporary address:}
\\ Departamento de Matem\' atica
\\ Faculdade de Ci\^encias da Universidade do Porto
\\ Rua do Campo Alegre 687
\\ 4169-007 Porto, Portugal
\\ www.fc.up.pt, \verb|crito@fc.up.pt|

\bigskip

\noindent Xavier Roulleau
\vspace{0.1cm}
\\Aix-Marseille Universit\'e, CNRS, Centrale Marseille,
\\I2M UMR 7373, 
\\13453 Marseille, France
\\ \verb|Xavier.Roulleau@univ-amu.fr|

\end{document}